\documentclass[leqno,11pt]{article}%
\usepackage{amsfonts}
\usepackage{amsmath}
\usepackage{amssymb}
\usepackage{palatino}
\usepackage[colorlinks]{hyperref}
\usepackage{graphicx}%
\providecommand{\U}[1]{\protect\rule{.1in}{.1in}}
\hypersetup{colorlinks=true, urlcolor=red, linkcolor=blue, citecolor=blue}
\newtheorem{theorem}{Theorem}[section]
\newenvironment{acknowledgement}
{\par\medskip\noindent\textbf{Acknowledgement.}\ }
{\par\medskip}

\newtheorem{condition}[theorem]{Assumption}

\newtheorem{definition}{Definition}[section]
\newtheorem{example}{Example}[section]

\newtheorem{lemma}{Lemma}[section]

\newtheorem{problem}[theorem]{Problem}
\newtheorem{proposition}{Proposition}[section]
\newtheorem{remark}{Remark}[section]

\newenvironment{proof}[1][Proof]{\noindent\textbf{#1.} }{\ \rule{0.5em}{0.5em}}
\begin{document}

\title{C\`{a}dl\`{a}g Solutions to Backward Stochastic Dynamics featuring Oblique
Subgradients and driven by Martingale Noise\thanks{E-mails:
grajdeanuandreea19@gmail.com (Andreea Negru\c{t}), aurel.rascanu@uaic.ro
(Aurel R\u{a}\c{s}canu), eduard.rotenstein@uaic.ro (Eduard Rotenstein)\newline%
$\sharp$~corresponding author}}
\author{Andreea Negru\c{t}$^{a}$, Aurel R\u{a}\c{s}canu$^{b,c}$, Eduard
Rotenstein$^{a,b,\sharp}$\medskip\\$^{a}${\small Simion Stoilow Institute of Mathematics of the Romanian Academy,
}\\{\small 21 Calea Grivi\c{t}ei, Bucharest, Rom\^{a}nia}\\$^{b}${\small Faculty of Mathematics, "Alexandru Ioan Cuza" University of
Ia\c{s}i, }\\{\small 9 Carol I Blvd., Ia\c{s}i,} {\small Rom\^{a}nia}\\$^{c}${\small Octav Mayer Institute of Mathematics of the Romanian Academy,
Ia\c{s}i branch, }\\{\small Bd. Carol I no. 8, Rom\^{a}nia}}
\date{}
\maketitle

\begin{abstract}
The present study improves the qualitative analysis of backward stochastic
variational dynamics on a general complete filtered probability space,
considered in the spirit of Liang, Lyons and Qian (2011). Our primary
objective is to overcome a substantial limitation in the study of Bensoussan,
Li and Yam (2018), where the boundedness condition imposed on the multivalued
subdifferential operator excludes standard obstacle-type constraints and
indicator functions of convex sets. We prove the existence and uniqueness of a
strong c\`{a}dl\`{a}g solution under the natural assumption that the driving
proper lower semicontinuous convex function is merely bounded from below by an
affine/quadratic function. Furthermore, we incorporate an oblique reflection
governed by a time-dependent, uniformly positive definite symmetric matrix, in
the spirit of the pioneering results of Gassous, R\u{a}\c{s}canu and
Rotenstein (2012, 2015).

\end{abstract}


\textbf{MSC2020 Subject Classification: }60H10, 60H30, 49K45

\section{Introduction. A motivating obstacle problem}

Since their introduction into the scientific literature, BSDEs have become an
important topic of research, both from the perspective of generalizing and
extending the underlying framework of the problems and from that of developing
applications and establishing connections with other areas of study, such as
(deterministic) PDEs, control problems, and convex and nonconvex optimization
problems. Given the scope of the present note, we do not aim to provide a
comprehensive account of the historical development of the problem. Instead,
we focus directly on the particular issue addressed in this work. In 2011,
Liang, Lyons and Qian \cite{Liang/Lyons/Qian:11} introduced a novel framework
for the study of backward stochastic dynamics on a general complete filtered
probability space, in which the filtration is not required to be generated by
a Brownian motion. The authors show that the existence and uniqueness of a
solution in a suitable space is equivalent to solving a functional
differential equation on certain path spaces. An important feature of this
framework is that neither It\^{o} integration nor martingale representation
formulas are required. In particular, the role of the $Z$-term of a classical
BSDE is taken over by a suitable martingale, which can be identified through
the generator by means of an appropriately chosen functional transformation.

To illustrate this idea, consider a filtered probability space $\left(
\Omega,\mathcal{F},\mathbb{P},\mathcal{F}_{t},B_{t}\right)  _{t\geq0}$, where
$\left\{  B_{t}:t\geq0\right\}  $ is a $d$-dimensional Brownian motion, and
the BSDE
\begin{equation}
\left\{
\begin{array}
[c]{l}%
-dY_{t}=F(t,Y_{t},Z_{t})dt-Z_{t}dB_{t},\quad\ t\!\in\![0,T],\medskip\\
Y_{T}=\eta\in L^{2}({\Omega},\mathcal{F}_{T},\mathbb{P};\mathbb{R}^{d}).
\end{array}
\right.  \label{one}%
\end{equation}
Assume that the generator $F$ satisfies conditions ensuring the existence of a
unique solution $\left(  Y,Z\right)  $. The process $Z$ is then uniquely
determined through the martingale representation $M_{t}=\int_{0}^{t}%
Z_{s}dB_{s}$. Writing $Z_{t}=\mathcal{R}_{t}(M)$, equation (\ref{one}) becomes%
\begin{equation}
Y_{t}=\eta+\int_{t}^{T}F(r,Y_{r},\mathcal{R}_{r}\left(  M\right)  )dr-\left(
M_{T}-M_{t}\right)  ,\quad\text{a.s., for all }t\in\left[  0,T\right]  .
\label{bsde-MF}%
\end{equation}
Define $V_{t}:=%
%TCIMACRO{\dint _{0}^{t}}%
%BeginExpansion
{\displaystyle\int_{0}^{t}}
%EndExpansion
F(r,Y_{r},\mathcal{R}_{r}\left(  M\right)  )dr-Y_{0}.$ Then $Y_{t}=M_{t}%
-V_{t}$. Since $Y_{T}=\eta=M_{T}-V_{T}$, it follows that%
\begin{equation}%
\begin{array}
[c]{c}%
M_{t}=\mathbb{E}^{\mathcal{F}_{t}}M_{T}=\mathbb{E}^{\mathcal{F}_{t}}\left(
\eta+V_{T}\right)  ,\quad\text{and}\quad Y_{t}=\mathbb{E}^{\mathcal{F}_{t}%
}\left(  \eta+V_{T}\right)  -V_{t}%
\end{array}
. \label{two}%
\end{equation}
Consequently, $V$ satisfies the functional differential equation%
\[
dV_{t}=F\left(  t,\mathbb{E}^{\mathcal{F}_{t}}\left(  \eta+V_{T}\right)
-V_{t},\mathcal{R}_{t}\left(  \mathbb{E}^{\mathcal{F}_{\cdot}}\left(
\eta+V_{T}\right)  \right)  \right)  dt
\]
and the solution $\left(  Y,M\right)  $ of equation (\ref{bsde-MF}) is given
by (\ref{two}).

More recently, Bensoussan, Li, and Yam \cite{Bensoussan/Li/Yam:2018}
considered, within the same framework, backward stochastic dynamics involving
a multivalued differential operator of subdifferential type. They also
established a detailed correspondence between the resulting stochastic
variational inequalities and weak solutions --- rather than viscosity
solutions, due to the intrinsic nonlocal nature of the integral involving the
gradient --- of a class of nonlocal parabolic variational inequalities and
parabolic partial differential equations, respectively.

However, the existence result for the multivalued backward stochastic dynamics
established in \cite{Bensoussan/Li/Yam:2018} relies on a rather restrictive
assumption on the subdifferential operator. More precisely, Assumption
$(H4-(iii))$ in Section 4 essentially requires the underlying convex function
to be bounded on bounded sets. This condition excludes several important
classes of multivalued equations, most notably obstacle-type problems. Indeed,
it is well known that if $E$ is a closed convex subset of $\mathbb{R}^{d}$,
then the convexity indicator function $I_{E}$ of $E$, i.e. $\varphi\left(
x\right)  =I_{E}\left(  x\right)  =0$ for $x\in E$ and $I_{E}\left(  x\right)
=+\infty$ otherwise, is a proper convex lower semicontinuous function, and,
for $x\in E$,%
\[
\partial I_{E}\left(  x\right)  =\{\hat{x}\in\mathbb{R}^{d}:\left\langle
\hat{x},y-x\right\rangle \leq0,\ \forall y\in E\}=N_{E}\left(  x\right)  ,
\]
where $N_{E}\left(  x\right)  $ denotes the (closed) outward normal cone to
$E$ at $x$. Moreover, $N_{E}\left(  x\right)  =\emptyset$ if $x\notin E$,
while $N_{E}\left(  x\right)  =\{0\}$ if $x\in\mathrm{int}\left(  E\right)  $,
the interior of $E$. Such a function is clearly unbounded on every bounded set
that meets the complement of $E$.

The type of equation considered in this paper belongs to the class of the
following \textit{obliquely reflected} backward stochastic differential
inclusion:%
\begin{equation}
-dY_{t}+H(t,Y_{t})\partial\varphi\left(  Y_{t}\right)  dt\ni F(t,Y_{t}%
,\mathcal{R}(M)_{t})dt-dM_{t},\quad t\in\left[  0,T\right]  ,\quad Y_{T}=\eta,
\label{eq to study}%
\end{equation}
where $H$ is a time-dependent regular matrix acting on the set of subgradients
and $\varphi$ is a proper lower semicontinuous convex function. The left-hand
side of Eq.(\ref{eq to study}) should be understood as it was introduced in
the first studies on this topic by Gassous, R\u{a}\c{s}canu and Rotenstein
(\cite{Gassous/Rascanu/Rotenstein:12}, \cite{Gassous/Rascanu/Rotenstein:15}).
A major difficulty in treating this class of equations stems from the rather
strong assumption imposed on the Lipschitz perturbation $H$, namely, its
symmetry. Removing this assumption is beyond the scope of the present work, as
it would require a substantially different approach. We therefore retain this
structural assumption and concentrate instead on relaxing the conditions
imposed on the multivalued operator and extending the framework to generalized
reflection at the boundary of the domain.

The paper is organized as follows. In Section \ref{Working setup}, we
introduce the underlying probabilistic framework and state the technical
assumptions imposed on the coefficients of Eq.(\ref{eq to study}). We also
recall the main tools required for the c\`{a}dl\`{a}g setting considered here,
which was introduced, although not fully developed, in
\cite{Bensoussan/Li/Yam:2018}. In Section \ref{main results section}, we first
establish some a priori estimates for the approximating sequence associated
with Eq.(\ref{eq to study}), under more natural and less restrictive
assumptions on $\varphi$. We then turn to the case of a possibly unbounded
obliquely reflected multivalued operator and prove the existence and
uniqueness of a solution within the c\`{a}dl\`{a}g framework. The last
section, Annex brings together some instruments and tools used along our study.

\section{Preliminaries\label{Working setup}}

Throughout the paper we work on a complete filtered probability space $\left(
\Omega,\mathcal{F},\mathbb{F}=\{\mathcal{F}_{t}\}_{t\geq0},\mathbb{P}\right)
$ satisfying the usual hypotheses, and on a fixed time interval $\left[
0,T\right]  $, $T>0$. We shall use the following spaces of stochastic processes:

\begin{enumerate}
\item[$\left(  a\right)  $] $\mathbb{L}_{m\times d}^{p},$ $p\geq0,$
$m,d\in\mathbb{N}^{\ast}$, is the (non-separable) complete metric space of
adapted c\`{a}gl\`{a}d processes $G:\Omega\times\left[  0,T\right]
\rightarrow\mathbb{R}^{m\times d}$ (\textit{left continuous with right limits;
abbreviated from the French }"\textit{continue \`{a} gauche, limite \`{a}
droite"}), and $\mathbb{D}_{d}^{p}$ is the complete metric space of adapted
c\`{a}dl\`{a}g processes $X:\Omega\times\left[  0,T\right]  \rightarrow
\mathbb{R}^{d}$ (\textit{right continuous with left limits; abbreviated from
}"\textit{continue \`{a} droite, limite \`{a} gauche"}). In both cases the
metric is defined by%
\[
\rho\left(  X,Y\right)  =\left\{
\begin{array}
[c]{ll}%
\left(  \mathbb{E}\sup\limits_{t\in\left[  0,T\right]  }\left\vert X_{t}%
-Y_{t}\right\vert ^{p}\right)  ^{1\wedge\left(  1/p\right)  }\quad & \text{if
}p>0,\medskip\\
\mathbb{E}\left(  1\wedge\sup\limits_{t\in\left[  0,T\right]  }\left\vert
X_{t}-Y_{t}\right\vert \right)  \quad & \text{if }p=0.
\end{array}
\right.
\]
If $p\geq1$, then the spaces are Banach spaces with the norm $\left\Vert
X\right\Vert =$ $\rho\left(  X,0\right)  .$\newline In the case $\mathbb{D}%
_{d}^{2}$ we denote%
\[
\left\vert \left\vert \left\vert X\right\vert \right\vert \right\vert
_{T}=\left\vert \left\vert \left\vert X\right\vert \right\vert \right\vert
_{\left[  0,T\right]  }:=\mathbb{E}\sup_{t\in\left[  0,T\right]  }\left\vert
X_{t}\right\vert ^{2}<\infty.
\]
(The inner product in $\mathbb{R}^{d}$ will be denoted by $\left\langle
\cdot,\cdot\right\rangle $ and the induced norm by $\left\vert \cdot
\right\vert ;$ for two matrices $X=\left(  x_{i,j}\right)  _{m\times d}%
\in\mathbb{R}^{m\times d}$ and $Y=\left(  y_{i,j}\right)  _{m\times d}$
$\in\mathbb{R}^{m\times d}$, the inner product is defined by $\left\langle
X,Y\right\rangle :=\mathrm{Tr}\left(  X^{\ast}Y\right)  =\sum_{i=1}^{m}%
\sum_{j=1}^{d}x_{i,j}y_{i,j}$ and consequently the corresponding norm is the
Frobenius norm $\left\vert Y\right\vert :=\left(  \mathrm{Tr}\left(  Y^{\ast
}Y\right)  \right)  ^{1/2}=\left(  \sum_{i,j}\left\vert y_{i,j}\right\vert
^{2}\right)  ^{1/2}\;$). The operator norm is defined as $\left\Vert
Y\right\Vert _{op}:=\sup_{\left\vert u\right\vert \neq0}\left\vert
Yu\right\vert /\left\vert u\right\vert $; it holds that $\dfrac{1}{\sqrt{d}}$
$\left\vert Y\right\vert \leq\left\Vert Y\right\Vert _{op}\leq\left\vert
Y\right\vert .$

\item[$\left(  b\right)  $] $\mathcal{M}_{d}^{2}\subset\mathbb{D}_{d}^{2}$ is
the Hilbert space of stochastic processes $M:\Omega\times\left[  0,T\right]
\rightarrow\mathbb{R}^{d},$ $M_{0}=0,$ which are c\`{a}dl\`{a}g square
integrable martingales on $\left[  0,T\right]  $, endowed with the inner
product and the corresponding norm
\[
\left\langle M,N\right\rangle _{\mathcal{M}}=\mathbb{E}\left\langle
M_{T},N_{T}\right\rangle \quad\text{and}\quad\left\Vert M\right\Vert
_{\mathcal{M}}=\sqrt{\mathbb{E}\left\vert M_{T}\right\vert ^{2}}.
\]
Since the filtration $\mathbb{F}=\{\mathcal{F}_{t}\}_{t\geq0}$ satisfies the
usual hypotheses, every martingale admits a unique c\`{a}dl\`{a}g modification
(see \cite[Chapter I, Theorem 9]{Protter:05}). In what follows, each
martingale is identified with its c\`{a}dl\`{a}g version.

The space $\mathcal{M}_{d}^{2}$ is a closed linear subspace of the Banach
space $\mathbb{D}_{d}^{2}$ since, by Doob's maximal $L^{2}$-inequality and the
fact that $M_{0}=0$, the norms$\ \left\Vert \cdot\right\Vert _{2}$ and
$\left\vert \left\vert \left\vert \cdot\right\vert \right\vert \right\vert
_{\left[  0,T\right]  }$ are equivalent:%
\begin{equation}
\mathbb{E}\left\vert M_{t}\right\vert ^{2}\leq\mathbb{E}\sup_{r\in\left[
0,t\right]  }\left\vert M_{r}\right\vert ^{2}\leq4\mathbb{E}\left\vert
M_{t}\right\vert ^{2},\quad\forall t\in\left[  0,T\right]  .
\label{bdg adapted}%
\end{equation}
In particular,%
\[
\left\Vert M\right\Vert _{2}\leq\left\vert \left\vert \left\vert M\right\vert
\right\vert \right\vert _{\left[  0,T\right]  }\leq2\left\Vert M\right\Vert
_{2}~.
\]
The interested reader is invited to consult, for instance, Pardoux and
R\u{a}\c{s}canu \cite[Proposition 1.56]{Pardoux/Rascanu:14} with the remark
that the result there, with the same proof, also holds if the assumption that
$X$ has continuous trajectories is replaced by right-continuous trajectories.

For $M\in\mathcal{M}_{d}^{2}$ we denote by $\left\langle M\right\rangle $ the
\textit{predictable quadratic variation} of $M$: the unique predictable,
c\`{a}dl\`{a}g, nondecreasing process with $\left\langle M\right\rangle
_{0}=0$ such that
\begin{equation}
\left\vert M_{t}\right\vert ^{2}-\left\langle M\right\rangle _{t}\quad\text{is
a local martingale;} \label{def-angle}%
\end{equation}
its existence and uniqueness are given by the Doob--Meyer decomposition of the
submartingale $\left\vert M\right\vert ^{2}$. For $M,N\in\mathcal{M}_{d}^{2}$
one sets $\left\langle M,N\right\rangle :=\frac{1}{4}\big(\left\langle
M+N\right\rangle -\left\langle M-N\right\rangle \big)$. \newline The quadratic
variation of $M$ is the c\`{a}dl\`{a}g, nondecreasing, adapted process defined
by%
\begin{equation}%
\begin{array}
[c]{rl}%
\lbrack M]_{t}:= & \text{(\textit{up)-}}\lim\limits_{n\rightarrow\infty
}\left(
%TCIMACRO{\dsum \limits_{k=0}^{n-1}}%
%BeginExpansion
{\displaystyle\sum\limits_{k=0}^{n-1}}
%EndExpansion
\left\vert M_{t\wedge t_{k+1}}-M_{t\wedge t_{k}}\right\vert ^{2}\right)  ,
\end{array}
\label{qv-def-m}%
\end{equation}
where $t_{k}=kT/n~$and the convergence is in probability uniform with respect
to $t\in\left[  0,T\right]  $. Taking expectations in (\ref{def-angle}) and
(\ref{qv-def-m}), it follows that
\begin{equation}
\mathbb{E}\left\langle M\right\rangle _{t}=\mathbb{E}[M]_{t}=\mathbb{E}%
\left\vert M_{t}\right\vert ^{2},\qquad t\in\left[  0,T\right]  .
\label{angle-energy}%
\end{equation}

\item[$\left(  c\right)  $] $\Lambda_{d\times k}^{2}$ (and $\Lambda_{d}%
^{2}:=\Lambda_{d\times1}^{2}$) is the Hilbert space of $\mathcal{F}_{t}%
$-progressively measurable, $\mathbb{R}^{d\times k}$-valued stochastic
proce\-sses $X,Y:\Omega\times\left[  0,T\right]  \rightarrow\mathbb{R}%
^{d\times k}$ such that $\mathbb{E}\int_{0}^{T}\left\vert X_{r}\right\vert
^{2}dr<\infty$, equipped with the norm $\left\Vert \cdot\right\Vert
_{\Lambda_{d}^{2}\left[  0,T\right]  }$ induced by the inner product%
\[
\left\langle X,Y\right\rangle _{\Lambda}:=\mathbb{E}%
%TCIMACRO{\dint _{0}^{T}}%
%BeginExpansion
{\displaystyle\int_{0}^{T}}
%EndExpansion
\mathrm{Tr}\left(  X_{r}^{\ast}Y_{r}\right)  dr.
\]

\end{enumerate}

We note that $\mathbb{D}_{d}^{2}\subset\Lambda_{d}^{2}$.

A c\`{a}dl\`{a}g stochastic process $X$ is a semimartingale if $X$ can be
represented as the sum $X=V+M,$ where $V$ is an adapted c\`{a}dl\`{a}g process
of bounded variation on compact intervals and $M$ is a c\`{a}dl\`{a}g local
martingale with $M_{0}=0$. If, moreover, $V$ is a predictable\footnote{A
process $K$ is predictable if its value at time $t$ is already determined by
the information available strictly before $t$, i.e. if $\left(  \omega
,t\right)  \longmapsto K\left(  \omega,t\right)  $ is measurable with respect
to the $\sigma$-algebra generated by all left-continuous adapted processes.}
(for instance, left-continuous and adapted) stochastic process, then $X$ is
called a special semimartingale.

\begin{condition}
[H1]\label{H1}The terminal datum $\eta\in L^{2}(\Omega,\mathcal{F}%
_{T},\mathbb{P};\mathbb{R}^{d})$.
\end{condition}

\begin{condition}
[H2]\label{H2}The generator $F\left(  \cdot,\cdot,y,z\right)  :\Omega
\times\left[  0,T\right]  \rightarrow\mathbb{R}^{d}$ is $\mathcal{F}_{t}%
$-progressively measurable for every $\left(  y,z\right)  \in\mathbb{R}%
^{d}\times\mathbb{R}^{d\times k}$, and there exist $L,\ell\in L^{2}\left(
0,T;\mathbb{R}_{+}\right)  $ such that:

\begin{itemize}
\item[$\left(  i\right)  $] \textit{Lipschitz conditions}: for all
$y,y^{\prime}\in\mathbb{R}^{d},\;z,z^{\prime}\in\mathbb{R}^{d\times
k},\;d\mathbb{P}\otimes dt$-$a.e.:$%
\begin{equation}
\left\vert F(t,y^{\prime},z)-F(t,y,z)\right\vert \leq L\left(  t\right)
|y^{\prime}-y|\quad\text{and}\quad|F(t,y,z^{\prime})-F(t,y,z)|\leq\ell\left(
t\right)  |z^{\prime}-z| \label{lip_F}%
\end{equation}

\item[$\left(  ii\right)  $] \textit{Boundedness condition}: $\mathbb{E}%
%TCIMACRO{\dint \nolimits_{0}^{T}}%
%BeginExpansion
{\displaystyle\int\nolimits_{0}^{T}}
%EndExpansion
\left\vert F\left(  t,0,0\right)  \right\vert ^{2}dt<+\infty$.
\end{itemize}
\end{condition}

\begin{condition}
[H3]\label{H3}$\mathcal{R}:\mathcal{M}_{d}^{2}\longrightarrow\Lambda_{d\times
k}^{2}\quad$is a mapping such that:

\begin{itemize}
\item[$\left(  j\right)  $] $\mathcal{R}(M)\equiv0,$ if $M\equiv0,$
$\mathbb{P}$-a.s.;

\item[$\left(  jj\right)  $] there exists $C_{\mathcal{R}}>0$ such that, for
any $0\leq s<t\leq T$ and for any $M,N\in\mathcal{M}_{d}^{2}~,$%
\begin{equation}
\mathbb{E}%
%TCIMACRO{\dint _{s}^{t}}%
%BeginExpansion
{\displaystyle\int_{s}^{t}}
%EndExpansion
\left\vert \mathcal{R}_{r}(M)-\mathcal{R}_{r}(N)\right\vert ^{2}dr\leq
C_{\mathcal{R}}^{2}\,\mathbb{E}%
%TCIMACRO{\dint _{s+}^{t}}%
%BeginExpansion
{\displaystyle\int_{s+}^{t}}
%EndExpansion
d\left[  M-N\right]  _{r}, \label{lip_R}%
\end{equation}
where $\left[  M-N\right]  $ is the quadratic variation (\ref{qv-def-m}) (or,
more explicitly, (\ref{qv})) of the martingale $M-N$, a c\`{a}dl\`{a}g
nondecreasing adapted process.
\end{itemize}
\end{condition}

Since the increments of a square integrable martingale are orthogonal in
$L^{2}$, the right-hand side of (\ref{lip_R}) may equally be written (notice
that the second equality from below is valid only under expectation, it can
not be interpreted pathwise) as:%
\begin{equation}
\mathbb{E}\int_{s+}^{t}d\left[  M-N\right]  _{r}=\mathbb{E}\left[  M-N\right]
_{t}-\mathbb{E}\left[  M-N\right]  _{s}=\mathbb{E}\left\vert M_{t}%
-N_{t}\right\vert ^{2}-\mathbb{E}\left\vert M_{s}-N_{s}\right\vert ^{2},
\label{brk incr}%
\end{equation}
being the form in which $\left(  \ref{H3}-\left(  jj\right)  \right)  $ is
verified in the examples below. Assumption $\left(  \ref{H3}-\left(
jj\right)  \right)  $ states precisely that, for all $M,N\in\mathcal{M}%
_{d}^{2}~$, the deterministic measure $\mathbb{E}\big|\mathcal{R}%
_{r}(M)-\mathcal{R}_{r}(N)\big|^{2}dr$ is absolutely continuous with respect
to $\mathbb{E}d\left[  M-N\right]  _{r}~$, with density bounded by
$C_{\mathcal{R}}^{2}$. It is therefore a \textit{locality} requirement:
$\mathcal{R}_{r}$ must be governed by the quadratic variation of $M$ at time
$r$.

Let us observe that $\left(  jj\right)  $, written for $s=0$ and $t=T$,
already gives%
\[
\big\Vert\mathcal{R}(M)-\mathcal{R}(N)\big\Vert_{\Lambda_{d\times k}^{2}}%
^{2}\leq C_{\mathcal{R}}^{2}\,\mathbb{E}\big|M_{T}-N_{T}\big|^{2}%
=C_{\mathcal{R}}^{2}\left\Vert M-N\right\Vert _{\mathcal{M}}^{2}~,
\]
i.e. $\mathcal{R}:\mathcal{M}_{d}^{2}\longrightarrow\Lambda_{d\times k}%
^{2}\quad$is a Lipschitz mapping.$\smallskip$

The following examples show that the class of such mappings is rich, and that
it is by no means restricted to the extraction of the integrand of an It\^{o}
integral with respect to a Brownian motion.

\begin{example}
\label{ex-normal-mart}\textbf{(Normal martingales: a driving noise which may
be purely discontinuous)}

In the introductory multidimensional Brownian setup, $\mathcal{R}$ was
obtained by inverting the map $Z\longmapsto%
%TCIMACRO{\tint _{0}^{\cdot}}%
%BeginExpansion
{\textstyle\int_{0}^{\cdot}}
%EndExpansion
Z_{s}dB_{s}$. What makes this work is not the Gaussian character of $B=\left(
B^{i}\right)  _{i}$, but the single identity $\left\langle B^{i}%
,B^{j}\right\rangle _{t}=\delta_{ij}t$: it is this identity that turns the
It\^{o} isometry into
\[
\mathbb{E}\left\vert \int_{0}^{t}Z_{s}\,dB_{s}\right\vert ^{2}=\mathbb{E}%
\int_{0}^{t}\left\vert Z_{s}\right\vert ^{2}ds,
\]
and it is this isometry, and nothing else, that produces the constant
$C_{\mathcal{R}}$ in Assumption \ref{H3}. Any martingale sharing that identity
will therefore serve equally well.$\smallskip$

A martingale $W=(W^{1},\ldots,W^{k})\in\mathcal{M}_{k}^{2}~,$ with $W_{0}=0$
is called a \textit{normal martingale} if
\begin{equation}
\left\langle W^{i},W^{j}\right\rangle _{t}=\delta_{ij}\,t,\qquad1\leq i,j\leq
k, \label{normal-mart}%
\end{equation}
that is, $W^{i}W^{j}-\delta_{ij}t$ is a martingale, for all $i,j$. Assume
$\mathbb{F=F}^{W}$ is the complete natural filtration generated by $W$, all
martingales being understood with respect to it. We note that, if every
$M\in\mathcal{M}_{d}^{2}$ can be written as%
\begin{equation}
M_{t}=\int_{0+}^{t}Z_{s}\,dW_{s},\qquad t\in\left[  0,T\right]  ,\;\mathbb{P}%
\text{-a.s.,} \label{prp}%
\end{equation}
with $Z\in\Lambda_{d\times k}^{2}$ being a predictable process, then we may
define $\mathcal{R}_{r}(M):=Z_{r}~$. According to (\ref{normal-mart}), the
process $Z$ is unique $d\mathbb{P}\otimes dt$-a.e.. Then, since%
\[
M-N=\int_{0+}^{\cdot}\left(  Z-Z^{\prime}\right)  dW=\int_{0+}^{\cdot}\left(
\mathcal{R}(M)-\mathcal{R}(N)\right)  dW,
\]
according to (\ref{ic}) we have:
\[
d\left[  M-N\right]  _{r}=\mathbf{Tr}\Big(\big(Z_{r}-Z_{r}^{\prime}%
\big)^{\ast}\big(Z_{r}-Z_{r}^{\prime}\big)\,d\left[  W,W\right]  _{r}\Big).
\]
The predictability of the matrix $\big(Z_{r}-Z_{r}^{\prime}\big)^{\ast
}\big(Z_{r}-Z_{r}^{\prime}\big)$ permits to $\left[  W,W\right]  $ to be
replaced by the expectation of its compensator $\left\langle W,W\right\rangle
$. Hence
\begin{equation}
\mathbb{E}\left[  M-N\right]  _{t}=\mathbb{E}\int_{0}^{t}\mathbf{Tr}%
\Big(\big(Z_{r}-Z_{r}^{\prime}\big)^{\ast}\big(Z_{r}-Z_{r}^{\prime
}\big)\Big)dr=\mathbb{E}\int_{0}^{t}\big|\mathcal{R}_{r}(M)-\mathcal{R}%
_{r}(N)\big|^{2}dr. \label{brk-normal}%
\end{equation}
In other words, $\left(  \left(  H\right)  _{3}-\left(  jj\right)  \right)  $
holds with equality and $C_{\mathcal{R}}=1$. We present below three situations
in which the predictable representation formula (\ref{prp}) holds:

\begin{itemize}
\item[$\left(  a\right)  $] $W=B$, a $k$-dimensional Brownian motion. This is
the classical case, and $\mathcal{R}(M)=Z\in\Lambda_{d\times k}^{2}$ is the
It\^{o} integrand;

\item[$\left(  b\right)  $] $W_{t}=\big(N_{t}-\lambda t\big)/\sqrt{\lambda}$,
with $N$ being a Poisson process of intensity $\lambda>0$, has the
representation property (\ref{prp}) with $Z\in\Lambda_{d\times1}^{2}$ being a
predictable process. Here, the driving noise is purely discontinuous and of
finite variation, and no Brownian motion occurs in the model;

\item[$\left(  c\right)  $] \textit{Az\'{e}ma martingales: }the normal
martingales that solve \'{E}mery's structure equation:
\[
d\left[  W\right]  _{t}=dt+\beta W_{t-}dW_{t}\text{,}\quad\text{with }\beta
\in\left[  -2,0\right]  ,
\]
have the representation property (\ref{prp}) with $Z\in\Lambda_{d\times1}^{2}%
$; see \'{E}mery \cite{Emery:89} and Attal, Belton \cite{Attal/Belton:07}. For
$\beta=-1$ one has the explicit representation structure:
\[
W_{t}=\sqrt{2}\,\mathrm{sgn}\left(  B_{t}\right)  \sqrt{t-g_{t}},\qquad
g_{t}:=\sup\left\{  s\leq t:B_{s}=0\right\}  ,
\]
$B$ being a Brownian motion and $\mathbb{F}$ the filtration generated by the
\emph{signs} of $B$. The process $W$ is neither Gaussian, nor a L\'{e}vy
process, nor of finite variation, and its natural filtration is strictly
smaller than the Brownian one.\vspace{-0.1in}
\end{itemize}

\smallskip\noindent\textbf{Two consequences.} First, remark that
$C_{\mathcal{R}}=1$ in all of the above. In consequence, the compatibility
condition from Theorem \ref{Main result1} is $a_{H}>\ell^{2}/2$, which is
exactly as in the Brownian case: allowing a discontinuous driving noise costs
nothing. Secondly, in the above situations $\left(  b\right)  $ and $\left(
c\right)  $ the martingale $M=\int Z\,dW$ introduces genuinely jumps. This is
precisely the situation for which the c\`{a}dl\`{a}g apparatus of the present
paper - the energy equality (\ref{ee}), the c\`{a}dl\`{a}g subdifferential
inequality (\ref{cadlag subdiff ineq}), the distinction, when necessary,
between $\int_{t}^{s}$ and $\int_{t+}^{s}$ - is required.
\end{example}

\begin{example}
\label{ex-jump}\textbf{(A nonlocal functional of the jump amplitudes)}

Consider $B$ a $k$-dimensional Brownian motion and let
\[
\hat{N}(dt,dx)=N(dt,dx)-dt\,\gamma(dx)
\]
be an independent compensated Poisson random measure on $\left[  0,T\right]
\times G$, where $\left(  G,\mathcal{G},\gamma\right)  $ is a $\sigma$-finite
measure space; let $\mathbb{F}$ be the completed filtration generated by $B$
and $N$. Every $M\in\mathcal{M}_{d}^{2}$ can be written as
\begin{equation}
M_{t}=\int_{0+}^{t}Z_{s}\,dB_{s}+\int_{0+}^{t}\int_{G}U_{s}(x)\,\hat
{N}(ds,dx), \label{levy-rep}%
\end{equation}
with $Z\in\Lambda_{d\times k}^{2}$ and $U$ a $d$-dimensional predictable
process such that $\mathbb{E}\int_{0}^{T}\!\int_{G}|U_{s}(x)|^{2}%
\gamma(dx)ds<\infty$. The two coefficients have quite different meanings:
$Z_{s}$ is the \textit{diffusive intensity} of $M$ at time $s$, whereas
$U_{s}(x)$ is the \textit{amplitude of the jump} that $M$ would perform at
time $s$ if a jump with mark $x$ occurred then. Accordingly, the quadratic
variation splits additively over the two components:
\begin{equation}
\mathbb{E}\left[  M\right]  _{t}=\mathbb{E}\left\vert M_{t}\right\vert
^{2}=\mathbb{E}\int_{0}^{t}\Big(\left\vert Z_{s}\right\vert ^{2}+\int
_{G}\left\vert U_{s}(x)\right\vert ^{2}\gamma(dx)\Big)ds. \label{brk-levy}%
\end{equation}
The linear mapping $\mathcal{R}:\mathcal{M}_{d}^{2}\rightarrow\Lambda_{d\times
k}^{2}\times\Lambda_{d}^{2}$
\[
\mathcal{R}_{r}(M):=\Big(aZ_{r},\int_{G}U_{r}(x)\beta(x)\gamma(dx)\Big),\quad
a\in\mathbb{R}\text{, }\beta\in L^{2}(G,\mathcal{G},\gamma;\mathbb{R}),
\]
satisfies Assumption \ref{H3} with the Lipschitz constant $C_{\mathcal{R}%
}:=\left\vert a\right\vert \vee\left\Vert \beta\right\Vert _{L^{2}(\gamma)}.$
We can mention the particular cases:

\begin{itemize}
\item[$\left(  a\right)  $] $\mathcal{R}_{r}(M):=Z_{r}$ (It\^{o}'s
representation) satisfies Assumption \ref{H3} with $C_{\mathcal{R}}=1$;

\item[$\left(  b\right)  $] $\mathcal{R}_{r}(M):=\int_{G}U_{r}(x)\,\beta
(x)\,\gamma(dx)\ \in\mathbb{R}^{d}$ satisfies Assumption \ref{H3} with
$C_{\mathcal{R}}=\left\Vert \beta\right\Vert _{L^{2}(\gamma)};$

\item[$\left(  c\right)  $] $\mathcal{R}_{r}(M)=\int_{G}U_{r}(x)\mathbf{1}%
_{A}\left(  x\right)  \gamma(dx),$ with $A\in\mathcal{G}$, $\gamma\left(
A\right)  <\infty$, satisfies Assumption \ref{H3} with $C_{\mathcal{R}}%
=\sqrt{\gamma\left(  A\right)  };$

\item[$\left(  d\right)  $] $\mathcal{R}_{r}(M)=\int_{G}\lambda U_{r}%
(x)\delta_{x_{0}}(dx)=\lambda U_{r}(x_{0})$ (the degenerate case $G=\{x_{0}%
\}$, $\gamma=\lambda\delta_{x_{0}}$ and $\beta\equiv1$) satisfies Assumption
\ref{H3} with $C_{\mathcal{R}}=\sqrt{\lambda}.$ In this case the measure
$\hat{N}$ reduces to a compensated Poisson process of intensity $\lambda$, and
$U_{r}:=U_{r}(x_{0})$ is $\mathbb{R}^{d}$-valued. The constant grows with the
jump intensity, so that the compatibility condition from Theorem
\ref{Main result1} becomes $a_{H}>\ell^{2}\lambda/2,$ for fixed $a_{H}$.
Therefore, high jump activity has to be offset by a small Lipschitz constant
$\ell$ of $F$, if considered in the $z$-variable.\vspace{-0.07in}
\end{itemize}

\smallskip\noindent\textbf{Why this example matters.} It is precisely a
dependence of this nonlocal type that gives rise, in Bensoussan, Li and Yam
\cite{Bensoussan/Li/Yam:2018}, to nonlocal parabolic variational inequalities;
the class of mappings covered by Assumption \ref{H3} thus reaches the very
situation which motivates the present study.
\end{example}

\begin{lemma}
\label{L1_gR}Assumption \ref{H3}$-\left(  jj\right)  $ is equivalent
to:\newline$\left(  jj^{\prime}\right)  \quad$There exists $C_{\mathcal{R}}>0$
such that, for every bounded Borel measurable function $g:\left[  0,T\right]
\rightarrow\mathbb{R}_{+}~$, all $M,N\in\mathcal{M}_{d}^{2}$ and all $0\leq
s<t\leq T$,%
\begin{equation}
\mathbb{E}\int_{s}^{t}g\left(  r\right)  \big|\mathcal{R}_{r}(M)-\mathcal{R}%
_{r}(N)\big|^{2}dr\leq C_{\mathcal{R}}^{2}\,\mathbb{E}\int_{(s,t]}g\left(
r\right)  d\left[  M-N\right]  _{r}~. \label{lip_gR}%
\end{equation}

\end{lemma}

\begin{proof}
$\left(  jj^{\prime}\right)  $ yields $\left(  jj\right)  $ for $g\equiv1.$
Let us prove $\left(  jj\right)  \Rightarrow\left(  jj^{\prime}\right)  .$
Inequality (\ref{lip_R}) guaranties that the deterministic nondecreasing
functions
\[
\alpha(t):=\int_{0}^{t}\mathbb{E~}\big|\mathcal{R}_{r}(M)-\mathcal{R}%
_{r}(N)\big|^{2}dr\quad\text{and}\quad\beta(t):=\mathbb{E}\left[  M-N\right]
_{t}%
\]
satisfy $d\alpha\left(  r\right)  \leq C_{\mathcal{R}}^{2}\,d\beta\left(
r\right)  ,$ as measures on $\left[  0,T\right]  $. As consequence,%
\[%
\begin{array}
[c]{l}%
\mathbb{E}%
%TCIMACRO{\dint _{s}^{t}}%
%BeginExpansion
{\displaystyle\int_{s}^{t}}
%EndExpansion
g\left(  r\right)  \big|\mathcal{R}_{r}(M)-\mathcal{R}_{r}(N)\big|^{2}dr=%
%TCIMACRO{\dint _{(s,t]}}%
%BeginExpansion
{\displaystyle\int_{(s,t]}}
%EndExpansion
g\left(  r\right)  d\alpha(r)\medskip\\
\quad\quad\quad\leq C_{\mathcal{R}}^{2}%
%TCIMACRO{\dint _{(s,t]}}%
%BeginExpansion
{\displaystyle\int_{(s,t]}}
%EndExpansion
g\left(  r\right)  d\beta\left(  r\right)  =C_{\mathcal{R}}^{2}\,\mathbb{E}%
%TCIMACRO{\dint _{(s,t]}}%
%BeginExpansion
{\displaystyle\int_{(s,t]}}
%EndExpansion
g\left(  r\right)  d\left[  M-N\right]  _{r}.
\end{array}
\]
The proof of the Lemma \ref{L1_gR} is now complete.\hfill\ 
\end{proof}

Under the above assumptions on $F$ and $\mathcal{R}$ we will deduce that%
\[
\left(  Y,M\right)  \longrightarrow\int_{0}^{\cdot}F(s,Y_{s},\mathcal{R}%
_{s}(M))ds
\]
is a continuous mapping from $\Lambda_{d}^{2}\times\mathcal{M}_{d}^{2}$ into
${\mathcal{{\mathbb{D}}}}_{d}^{2}$, as follows from the next lemma, whose
proof is a direct consequence of the Cauchy--Schwarz inequality together with
Assumptions \ref{H2} and \ref{H3} and is therefore omitted.

\begin{lemma}
\label{L2_F} For all $(Y,M)$, $\left(  Z,N\right)  \in\Lambda_{d}^{2}%
\times\mathcal{M}_{d}^{2}~$:
\[%
\begin{array}
[c]{l}%
\mathbb{E}\sup\limits_{t\in\left[  0,T\right]  }\left\vert
%TCIMACRO{\dint _{0}^{t}}%
%BeginExpansion
{\displaystyle\int_{0}^{t}}
%EndExpansion
F(s,Y_{s},\mathcal{R}_{s}(M))ds-%
%TCIMACRO{\dint _{0}^{t}}%
%BeginExpansion
{\displaystyle\int_{0}^{t}}
%EndExpansion
F(s,Z_{s},\mathcal{R}_{s}(N))ds\right\vert ^{2}\medskip\\
\quad\quad\quad\quad\leq\mathbb{E}\left(
%TCIMACRO{\dint _{0}^{T}}%
%BeginExpansion
{\displaystyle\int_{0}^{T}}
%EndExpansion
\left\vert F(s,Y_{s},\mathcal{R}_{s}(M))-F(s,Z_{s},\mathcal{R}_{s}%
(N))\right\vert ds\right)  ^{2}\medskip\\
\quad\quad\quad\quad\leq C\left(  \mathbb{E}%
%TCIMACRO{\dint _{0}^{T}}%
%BeginExpansion
{\displaystyle\int_{0}^{T}}
%EndExpansion
\left\vert Y_{s}-Z_{s}\right\vert ^{2}ds+\mathbb{E}~\left\vert M_{T}%
-N_{T}\right\vert ^{2}\right)  \medskip\\
\quad\quad\quad\quad\leq C\left[  T\times\mathbb{E}\sup\limits_{s\in\left[
0,T\right]  }\left\vert Y_{s}-Z_{s}\right\vert ^{2}+\mathbb{E}\left\vert
M_{T}-N_{T}\right\vert ^{2}\right]  ;
\end{array}
\]
where
\[
C:=2\int_{0}^{T}L^{2}\left(  s\right)  ds+2C_{\mathcal{R}}^{2}\int_{0}^{T}%
\ell^{2}\left(  s\right)  ds;
\]
in particular, for $\left(  Z,N\right)  =\left(  0,0\right)  $,%
\begin{equation}%
\begin{array}
[c]{l}%
\mathbb{E}\sup\limits_{t\in\left[  0,T\right]  }\left\vert
%TCIMACRO{\dint _{0}^{t}}%
%BeginExpansion
{\displaystyle\int_{0}^{t}}
%EndExpansion
F(s,Y_{s},\mathcal{R}_{s}(M))ds\right\vert ^{2}\leq\mathbb{E}\left(
%TCIMACRO{\dint _{0}^{T}}%
%BeginExpansion
{\displaystyle\int_{0}^{T}}
%EndExpansion
\left\vert F(s,Y_{s},\mathcal{R}_{s}(M))\right\vert ds\right)  ^{2}\medskip\\
\quad\quad\quad\quad\quad\leq2C\times\left(  \mathbb{E}%
%TCIMACRO{\dint _{0}^{T}}%
%BeginExpansion
{\displaystyle\int_{0}^{T}}
%EndExpansion
\left\vert Y_{s}\right\vert ^{2}ds+\mathbb{E}\left\vert M_{T}\right\vert
^{2}\right)  +2\mathbb{E}\left(
%TCIMACRO{\dint _{0}^{T}}%
%BeginExpansion
{\displaystyle\int_{0}^{T}}
%EndExpansion
\left\vert F(s,0,0)\right\vert ds\right)  ^{2}%
\end{array}
\label{le-F2}%
\end{equation}

\end{lemma}

We give now a preliminary result (similar to Theorem 4.1 from Liang, Lyons,
Qian \cite{Liang/Lyons/Qian:11}) concerning the behavior of the solution for a
BSDE driven by a martingale and featuring a functional representation of the
third variable of the driver $F$.

\begin{proposition}
\label{lip_FR}\textit{Under Assumptions \ref{H1}, \ref{H2}\ and \ref{H3},
there exists a unique pair }$(Y,M)\in\mathbb{D}_{d}^{2}\times\mathcal{M}%
_{d}^{2}~$\textit{, which is the solution of the equation }%
\begin{equation}
Y_{t}=\eta+%
%TCIMACRO{\dint _{t}^{T}}%
%BeginExpansion
{\displaystyle\int_{t}^{T}}
%EndExpansion
F\left(  r,Y_{r},\mathcal{R}_{r}(M)\right)  dr-(M_{T}-M_{t}),\text{ }%
t\in\left[  0,T\right]  ,\text{ a.s.,} \label{eq_FR}%
\end{equation}
\textit{and for some constant }$C=C(T,L,\ell,C_{\mathcal{R}})$\textit{, }%
\begin{equation}
\mathbb{E}\sup_{0\leq t\leq T}|Y_{t}|^{2}+\mathbb{E}\left\vert M_{T}%
\right\vert ^{2}\leq C\left(  \mathbb{E}|\eta|^{2}+\mathbb{E}\int_{0}%
^{T}|F(r,0,0)|^{2}dr\right)  . \label{estim FR}%
\end{equation}

\end{proposition}

\begin{proof}
\emph{Step 1. }Given $f:\Omega\times\left[  0,T\right]  \rightarrow
\mathbb{R}^{d}$ a progressively measurable stochastic process such that
\[
\mathbb{E}\left(
%TCIMACRO{\dint _{0}^{T}}%
%BeginExpansion
{\displaystyle\int_{0}^{T}}
%EndExpansion
\left\vert f_{r}\right\vert dr\right)  ^{2}<\infty,
\]
then the c\`{a}dl\`{a}g stochastic processes%
\[
\left\{
\begin{array}
[c]{l}%
\hat{X}_{t}=\mathbb{E}^{\mathcal{F}_{t}}\left(  \eta+%
%TCIMACRO{\dint _{0}^{T}}%
%BeginExpansion
{\displaystyle\int_{0}^{T}}
%EndExpansion
f_{r}dr\right)  -%
%TCIMACRO{\dint _{0}^{t}}%
%BeginExpansion
{\displaystyle\int_{0}^{t}}
%EndExpansion
f_{r}dr,\medskip\\
\hat{N}_{t}=\mathbb{E}^{\mathcal{F}_{t}}\left(  \eta+%
%TCIMACRO{\dint _{0}^{T}}%
%BeginExpansion
{\displaystyle\int_{0}^{T}}
%EndExpansion
f_{r}dr\right)  -\mathbb{E}^{\mathcal{F}_{0}}\left(  \eta+%
%TCIMACRO{\dint _{0}^{T}}%
%BeginExpansion
{\displaystyle\int_{0}^{T}}
%EndExpansion
f_{r}dr\right)
\end{array}
\right.
\]
clearly satisfy
\[
\hat{X}_{t}=\eta+%
%TCIMACRO{\dint _{t}^{T}}%
%BeginExpansion
{\displaystyle\int_{t}^{T}}
%EndExpansion
f_{r}dr-(\hat{N}_{T}-\hat{N}_{t}),\quad\text{for all }t\in\left[  0,T\right]
\text{, }\mathbb{P}\text{-a.s.}%
\]
Moreover, $(\hat{X},\hat{N})\in{\mathcal{{\mathbb{D}}}}_{d}^{2}\times
\mathcal{M}_{d}^{2}$ and
\begin{equation}
\mathbb{E}\sup_{t\in\left[  0,T\right]  }|\hat{X}_{t}|^{2}+\mathbb{E}|\hat
{N}_{T}|^{2}\leq16\left(  \mathbb{E}|\eta|^{2}+\mathbb{E}\left(  \int_{0}%
^{T}|f_{r}|dr\right)  ^{2}\right)  . \label{estim-1}%
\end{equation}
Indeed, by Doob's $L^{2}$-inequality%
\begin{align*}
\mathbb{E}\sup_{t\in\left[  0,T\right]  }|\hat{X}_{t}|^{2}  &  \leq
\mathbb{E}\sup_{t\in\left[  0,T\right]  }\left\vert \mathbb{E}^{\mathcal{F}%
_{t}}\left(  \left\vert \eta\right\vert +%
%TCIMACRO{\dint _{0}^{T}}%
%BeginExpansion
{\displaystyle\int_{0}^{T}}
%EndExpansion
\left\vert f_{r}\right\vert dr\right)  \right\vert ^{2}\\
&  \leq4\mathbb{E}\left(  |\eta|+\int_{0}^{T}|f_{r}|dr\right)  ^{2}%
\leq8\mathbb{E}|\eta|^{2}+8\mathbb{E}\left(  \int_{0}^{T}|f_{r}|dr\right)
^{2}<\infty
\end{align*}
and
\[
\mathbb{E}|\hat{N}_{T}|^{2}\leq2\mathbb{E}\left\vert \eta+%
%TCIMACRO{\dint _{0}^{T}}%
%BeginExpansion
{\displaystyle\int_{0}^{T}}
%EndExpansion
f_{r}dr\right\vert ^{2}+2\left\vert \mathbb{E}\left(  \eta+%
%TCIMACRO{\dint _{0}^{T}}%
%BeginExpansion
{\displaystyle\int_{0}^{T}}
%EndExpansion
f_{r}dr\right)  \right\vert ^{2}\leq8\mathbb{E}|\eta|^{2}+8\mathbb{E}\left(
\int_{0}^{T}|f_{r}|dr\right)  ^{2}%
\]
\emph{Step 2. }Let $\lambda\geq1$ ($\lambda$ will be specified below) and
\[
V\left(  t\right)  :=%
%TCIMACRO{\dint _{0}^{t}}%
%BeginExpansion
{\displaystyle\int_{0}^{t}}
%EndExpansion
\left[  1+L^{2}\left(  r\right)  +\ell^{2}\left(  r\right)  \right]  dr.
\]
On $\Lambda_{d}^{2}\times\mathcal{M}_{d}^{2}$ we consider the equivalent norm
\[
\Vert(Y,M)\Vert_{\lambda}^{2}:=\,\mathbb{E}\int_{0}^{T}e^{\lambda V\left(
r\right)  }|Y_{r}|^{2}dr+\int_{(0,T]}e^{\lambda V\left(  r\right)  }\,d\left[
M\right]  _{r}.
\]
The solution $\left(  Y,M\right)  $ of equation (\ref{eq_FR}) can be regarded
as a fixed point of the mapping%
\[
\Gamma:\Lambda_{d}^{2}\times\mathcal{M}_{d}^{2}\rightarrow\Lambda_{d}%
^{2}\times\mathcal{M}_{d}^{2}~,\quad(\hat{Y},\hat{M})=\Gamma(Y,M),
\]
where $(\hat{Y},\hat{M})\in{\mathcal{{\mathbb{D}}}}_{d}^{2}\times
\mathcal{M}_{d}^{2}$ is the solution of the BSDE%
\begin{equation}
\hat{Y}_{t}=\eta+%
%TCIMACRO{\dint _{t}^{T}}%
%BeginExpansion
{\displaystyle\int_{t}^{T}}
%EndExpansion
F\left(  r,Y_{r},\mathcal{R}_{r}(M)\right)  dr-(\hat{M}_{T}-\hat{M}_{t}),\quad
t\in\left[  0,T\right]  \text{, }\mathbb{P}\text{-a.s.} \label{1-F}%
\end{equation}
By Step 1,\emph{ }the mapping $\Gamma$ is well defined, since by Lemma
\ref{L2_F},%
\[
\mathbb{E~}\left(
%TCIMACRO{\dint _{0}^{T}}%
%BeginExpansion
{\displaystyle\int_{0}^{T}}
%EndExpansion
\left\vert F(s,Y_{s},\mathcal{R}_{s}(M))\right\vert ds\right)  ^{2}<\infty.
\]
Moreover, from Step 1, $\Gamma(Y,M)=(\hat{Y},\hat{M})\in{\mathcal{{\mathbb{D}%
}}}_{d}^{2}\times\mathcal{M}_{d}^{2}~.$ We shall prove that $\Gamma$ is a contraction.

Let $\left(  Y,M\right)  ,$ $\left(  Z,N\right)  \in\Lambda_{d}^{2}%
\times\mathcal{M}_{d}^{2}$ and $(\hat{Y},\hat{M})=\Gamma(Y,M)$, $(\hat{Z}%
,\hat{N})=$ $\Gamma\left(  Z,N\right)  .$ Subtracting the corresponding
equations (\ref{1-F}), we have%
\[
\hat{Y}_{t}-\hat{Z}_{t}=%
%TCIMACRO{\dint _{t}^{T}}%
%BeginExpansion
{\displaystyle\int_{t}^{T}}
%EndExpansion
\left[  F\left(  r,Y_{r},\mathcal{R}_{r}(M)\right)  -F\left(  r,Z_{r}%
,\mathcal{R}_{r}(N)\right)  \right]  dr-\left[  \big(\hat{M}_{T}-\hat{N}%
_{T}\big)-\big(\hat{M}_{t}-\hat{N}_{t}\big)\right]  .
\]
By It\^{o}'s formula applied to $e^{\lambda V\left(  t\right)  }\left\vert
\hat{Y}_{t}-\hat{Z}_{t}\right\vert ^{2}$, for which $\frac{d}{dr}e^{\lambda
V\left(  r\right)  }=\lambda\big(1+L^{2}\left(  r\right)  +\ell^{2}\left(
r\right)  \big)e^{\lambda V\left(  r\right)  }$, we obtain%
\begin{equation}%
\begin{array}
[c]{l}%
\mathbb{E}|\hat{Y}_{0}-\hat{Z}_{0}|^{2}+\lambda\mathbb{E}%
%TCIMACRO{\dint _{0}^{T}}%
%BeginExpansion
{\displaystyle\int_{0}^{T}}
%EndExpansion
\left(  1+L^{2}\left(  r\right)  +\ell^{2}\left(  r\right)  \right)
e^{\lambda V\left(  r\right)  }\left\vert \hat{Y}_{r}-\hat{Z}_{r}\right\vert
^{2}dr+\mathbb{E}{%
%TCIMACRO{\dint _{(0,T]}}%
%BeginExpansion
{\displaystyle\int_{(0,T]}}
%EndExpansion
}e^{\lambda V\left(  r\right)  }\,d\left[  \hat{M}-\hat{N}\right]
_{r}\medskip\\
\quad\quad\quad\quad=2\mathbb{E}%
%TCIMACRO{\dint _{0}^{T}}%
%BeginExpansion
{\displaystyle\int_{0}^{T}}
%EndExpansion
e^{\lambda V\left(  r\right)  }\left\langle \hat{Y}_{r}-\hat{Z}_{r},F\left(
r,Y_{r},\mathcal{R}_{r}(M)\right)  -F\left(  r,Z_{r},\mathcal{R}%
_{r}(N)\right)  \right\rangle dr.
\end{array}
\label{2-F}%
\end{equation}
The right-hand side is estimated as follows:%
\[%
\begin{array}
[c]{l}%
2\mathbb{E}%
%TCIMACRO{\dint _{0}^{T}}%
%BeginExpansion
{\displaystyle\int_{0}^{T}}
%EndExpansion
e^{\lambda V\left(  r\right)  }\left\langle \hat{Y}_{r}-\hat{Z}_{r},F\left(
r,Y_{r},\mathcal{R}_{r}(M)\right)  -F\left(  r,Z_{r},\mathcal{R}%
_{r}(N)\right)  \right\rangle dr\medskip\\
\quad\quad\leq2\mathbb{E}%
%TCIMACRO{\dint _{0}^{T}}%
%BeginExpansion
{\displaystyle\int_{0}^{T}}
%EndExpansion
e^{\lambda V\left(  r\right)  }|\hat{Y}_{r}-\hat{Z}_{r}|\left(  L\left(
r\right)  \left\vert Y_{r}-Z_{r}\right\vert +\ell\left(  r\right)  \left\vert
\mathcal{R}_{r}(M)-\mathcal{R}_{r}(N)\right\vert \right)  dr\medskip\\
\quad\quad\leq\mathbb{E}%
%TCIMACRO{\dint _{0}^{T}}%
%BeginExpansion
{\displaystyle\int_{0}^{T}}
%EndExpansion
e^{\lambda V\left(  r\right)  }\left(  \lambda\times L^{2}\left(  r\right)
|\hat{Y}_{r}-\hat{Z}_{r}|^{2}+\dfrac{1}{\lambda}\left\vert Y_{r}%
-Z_{r}\right\vert ^{2}\right)  dr\medskip\\
\quad\quad+\mathbb{E}%
%TCIMACRO{\dint _{0}^{T}}%
%BeginExpansion
{\displaystyle\int_{0}^{T}}
%EndExpansion
e^{\lambda V\left(  r\right)  }\left(  \lambda\times\ell^{2}\left(  r\right)
|\hat{Y}_{r}-\hat{Z}_{r}|^{2}+\dfrac{1}{\lambda}\left\vert \mathcal{R}%
_{r}(M)-\mathcal{R}_{r}(N)\right\vert ^{2}\right)  dr.
\end{array}
\]
Plugging this estimate into (\ref{2-F}), we see that the term $\lambda
\mathbb{E}\int_{0}^{T}e^{\lambda V\left(  r\right)  }\left(  L^{2}\left(
r\right)  +\ell^{2}\left(  r\right)  \right)  |\hat{Y}_{r}-\hat{Z}_{r}|^{2}dr$
cancel and, discarding $\mathbb{E~}|\hat{Y}_{0}-\hat{Z}_{0}|^{2}\geq0$, we
obtain for any $\lambda\geq1$,%
\[%
\begin{array}
[c]{l}%
\lambda\mathbb{E}%
%TCIMACRO{\dint _{0}^{T}}%
%BeginExpansion
{\displaystyle\int_{0}^{T}}
%EndExpansion
e^{\lambda V\left(  r\right)  }\left\vert \hat{Y}_{r}-\hat{Z}_{r}\right\vert
^{2}dr+%
%TCIMACRO{\dint _{(0,T]}}%
%BeginExpansion
{\displaystyle\int_{(0,T]}}
%EndExpansion
e^{\lambda V\left(  r\right)  }d\left[  \hat{M}-\hat{N}\right]  _{r}\medskip\\
\quad\quad\quad\quad\leq\dfrac{1}{\lambda}\mathbb{E}%
%TCIMACRO{\dint _{0}^{T}}%
%BeginExpansion
{\displaystyle\int_{0}^{T}}
%EndExpansion
e^{\lambda V\left(  r\right)  }\left\vert Y_{r}-Z_{r}\right\vert ^{2}%
dr+\dfrac{1}{\lambda}\mathbb{E}%
%TCIMACRO{\dint _{0}^{T}}%
%BeginExpansion
{\displaystyle\int_{0}^{T}}
%EndExpansion
e^{\lambda V\left(  r\right)  }\left\vert \mathcal{R}_{r}(M)-\mathcal{R}%
_{r}(N)\right\vert ^{2}dr.
\end{array}
\]
Using now Lemma \ref{L1_gR} for the last term and $\lambda\geq1$ for the first
term on the left-hand side, we deduce%
\[%
\begin{array}
[c]{l}%
\mathbb{E}%
%TCIMACRO{\dint _{0}^{T}}%
%BeginExpansion
{\displaystyle\int_{0}^{T}}
%EndExpansion
e^{\lambda V\left(  r\right)  }\left\vert \hat{Y}_{r}-\hat{Z}_{r}\right\vert
^{2}dr+%
%TCIMACRO{\dint _{(0,T]}}%
%BeginExpansion
{\displaystyle\int_{(0,T]}}
%EndExpansion
e^{\lambda V\left(  r\right)  }d\left[  \hat{M}-\hat{N}\right]  _{r}\medskip\\
\quad\quad\leq\dfrac{1+C_{\mathcal{R}}^{2}}{\lambda}\left(  \mathbb{E}%
%TCIMACRO{\dint _{0}^{T}}%
%BeginExpansion
{\displaystyle\int_{0}^{T}}
%EndExpansion
e^{\lambda V\left(  r\right)  }\left\vert Y_{r}-Z_{r}\right\vert ^{2}dr+%
%TCIMACRO{\dint _{(0,T]}}%
%BeginExpansion
{\displaystyle\int_{(0,T]}}
%EndExpansion
e^{\lambda V\left(  r\right)  }d\left[  \hat{M}-\hat{N}\right]  _{r}\right)
~,
\end{array}
\]
that is
\begin{equation}
\left\Vert (\hat{Y},\hat{M})-(\hat{Z},\hat{N})\right\Vert _{\lambda}^{2}%
\leq\dfrac{1}{4}\left\Vert (Y,M)-(Z,N)\right\Vert _{\lambda}^{2}~,
\label{3-Fym}%
\end{equation}
for $\lambda=4\left(  1+C_{\mathcal{R}}^{2}\right)  \geq1.$ Note that
$\lambda$ does \textit{not} depend on $T.\smallskip$

By the Banach fixed-point theorem there exists a unique pair $\left(
Y,M\right)  \in\Lambda_{d}^{2}\times\mathcal{M}_{d}^{2}$ such that $\left(
Y,M\right)  =\Gamma\left(  Y,M\right)  \in{\mathcal{{\mathbb{D}}}}_{d}%
^{2}\times\mathcal{M}_{d}^{2}$, that is, $\left(  Y,M\right)  $ is the unique
solution of the BSDE (\ref{eq_FR}). Let $(\hat{X},\hat{N})=\Gamma\left(
0,0\right)  $, that is%
\[
\hat{X}_{t}=\eta+%
%TCIMACRO{\dint _{t}^{T}}%
%BeginExpansion
{\displaystyle\int_{t}^{T}}
%EndExpansion
F\left(  r,0,0\right)  dr-(\hat{N}_{T}-\hat{N}_{t}).
\]
Since $\lambda\geq1$ and $dr\leq dV\left(  r\right)  $, we have%
\[%
\begin{array}
[c]{ccl}%
\mathbb{E}%
%TCIMACRO{\dint _{0}^{T}}%
%BeginExpansion
{\displaystyle\int_{0}^{T}}
%EndExpansion
e^{\lambda V\left(  r\right)  }|\hat{X}_{r}|^{2}dr & \leq & \mathbb{E}\left(
\sup\limits_{t\in\left[  0,T\right]  }|\hat{X}_{t}|^{2}\right)
%TCIMACRO{\dint _{0}^{T}}%
%BeginExpansion
{\displaystyle\int_{0}^{T}}
%EndExpansion
e^{\lambda V\left(  r\right)  }dV\left(  r\right)  \medskip\\
& = & \dfrac{e^{\lambda V\left(  T\right)  }-1}{\lambda}\mathbb{E}\left(
\sup\limits_{t\in\left[  0,T\right]  }|\hat{X}_{t}|^{2}\right)  \leq
e^{\lambda V\left(  T\right)  }\mathbb{E}\left(  \sup\limits_{t\in\left[
0,T\right]  }|\hat{X}_{t}|^{2}\right)  ,
\end{array}
\]
while%
\[
\mathbb{E}%
%TCIMACRO{\dint _{(0,T]}}%
%BeginExpansion
{\displaystyle\int_{(0,T]}}
%EndExpansion
e^{\lambda V\left(  r\right)  }d\left[  \hat{N}\right]  _{r}\leq e^{\lambda
V\left(  T\right)  }\mathbb{E}\left[  \hat{N}\right]  _{T}=e^{\lambda V\left(
T\right)  }\mathbb{E}|\hat{N}_{T}|^{2}.
\]
Hence, by (\ref{estim-1}),%
\begin{align*}
\left\Vert (\hat{X},\hat{N})\right\Vert _{\lambda}^{2}  &  \leq e^{\lambda
V\left(  T\right)  }\left(  \mathbb{E}\sup_{t\in\left[  0,T\right]  }|\hat
{X}_{t}|^{2}+\mathbb{E}|\hat{N}_{T}|^{2}\right) \\
&  \leq16e^{\lambda V\left(  T\right)  }\left(  \mathbb{E}|\eta|^{2}%
+\mathbb{E}\left(  \int_{0}^{T}|F\left(  r,0,0\right)  |dr\right)
^{2}\right)  .
\end{align*}
Then, from (\ref{3-Fym}) for the solution $\left(  Y,M\right)  $ and the above
estimate for $(\hat{X},\hat{N})$ we get%
\[%
\begin{array}
[c]{ccl}%
\left\Vert (Y,M)\right\Vert _{\lambda}^{2} & \leq & 2\left\Vert (Y,M)-(\hat
{X},\hat{N})\right\Vert _{\lambda}^{2}+2\left\Vert (\hat{X},\hat
{N})\right\Vert _{\lambda}^{2}\medskip\\
& \leq & \dfrac{1}{2}\left\Vert (Y,M)\right\Vert _{\lambda}^{2}+32e^{\lambda
V\left(  T\right)  }\left(  \mathbb{E}|\eta|^{2}+\mathbb{E}\left(
%TCIMACRO{\dint _{0}^{T}}%
%BeginExpansion
{\displaystyle\int_{0}^{T}}
%EndExpansion
|F\left(  r,0,0\right)  |dr\right)  ^{2}\right)  .
\end{array}
\]
Hence%
\begin{equation}
\mathbb{E}\int_{0}^{T}|Y_{r}|^{2}dr+\mathbb{E}\left\vert M_{T}\right\vert
^{2}\leq\left\Vert (Y,M)\right\Vert _{\lambda}^{2}\leq64e^{\lambda V\left(
T\right)  }\left(  \mathbb{E}|\eta|^{2}+\mathbb{E}\left(  \int_{0}%
^{T}|F\left(  r,0,0\right)  |dr\right)  ^{2}\right)  . \label{4-fym}%
\end{equation}
Since $Y_{t}=\mathbb{E}^{\mathcal{F}_{t}}(\eta+%
%TCIMACRO{\tint _{t}^{T}}%
%BeginExpansion
{\textstyle\int_{t}^{T}}
%EndExpansion
F\left(  r,Y_{r},\mathcal{R}_{r}(M)\right)  dr)$, then, as in Step 1, Doob's
$L^{2}$ inequality yields%
\begin{equation}
\mathbb{E}\sup_{t\in\left[  0,T\right]  }|Y_{t}|^{2}\leq8\mathbb{E}|\eta
|^{2}+8\mathbb{E}\left(  \int_{0}^{T}|F\left(  r,Y_{r},\mathcal{R}%
_{r}(M)\right)  |dr\right)  ^{2}. \label{5-fym}%
\end{equation}
Inserting the estimates (\ref{le-F2}) and (\ref{4-fym}) into (\ref{5-fym}), we
obtain the inequality (\ref{estim FR}).\hfill\medskip
\end{proof}

In what follows we revisit the results of Theorem 3.5 in Bensoussan, Li and
Yam \cite{Bensoussan/Li/Yam:2018}. Building on the estimates we will obtain in
the key Proposition \ref{Lemma with important bounds}, we prove that, within a
rigorous c\`{a}dl\`{a}g framework, the BSDE under consideration admits a
unique solution. Since the oblique reflection involves a perturbation matrix
$H$, as in the framework introduced by Gassous, R\u{a}\c{s}canu and Rotenstein
\cite{Gassous/Rascanu/Rotenstein:12}, \cite{Gassous/Rascanu/Rotenstein:15} and
Maticiuc and Rotenstein \cite{Maticiuc/Rotenstein:18}, we may employ a
suitable version of the standard approximation technique. In the backward
framework developed in \cite{Gassous/Rascanu/Rotenstein:15},
\cite{Maticiuc/Rotenstein:18}, the matrix-valued coefficient $H$ must be
independent of the state variable $Y$ in order to obtain a strong solution; if
$H$ depends on $Y$, the available arguments yield only a weak solution, a case
that lies beyond the scope of the present study. Accordingly, relying on the
estimates of Proposition \ref{Lemma with important bounds}, we show that the
approximating sequence is Cauchy and then identify its limit as the unique
solution of the equation.

\section{BSDEs with oblique subgradients and martingale noise\textit{
}\label{main results section}}

\subsection{Problem formulation}

We are interested in the study of the following problem.

\begin{problem}
\label{P}Show that there exists a unique triple $\left(  Y,M,U\right)  $,
belonging to a suitable space, such that%
\begin{equation}
\left\{
\begin{array}
[c]{l}%
Y_{t}+%
%TCIMACRO{\dint _{t}^{T}}%
%BeginExpansion
{\displaystyle\int_{t}^{T}}
%EndExpansion
H_{r}U_{r}dr=\eta+%
%TCIMACRO{\dint _{t}^{T}}%
%BeginExpansion
{\displaystyle\int_{t}^{T}}
%EndExpansion
F\left(  r,Y_{r},\mathcal{R}_{r}(M)\right)  dr-(M_{T}-M_{t}),\text{ }\forall
t\in\left[  0,T\right]  ,\text{ }\mathbb{P}\text{-}a.s.,\smallskip\\
U_{r}\in\partial\varphi\left(  Y_{r}\right)  ,\text{ }d\mathbb{P\otimes
}dr\text{-a.e.}%
\end{array}
\right.  \label{def of solution}%
\end{equation}
The exact meaning of the solution is given in Definition \ref{def of sol}.
\end{problem}

We assume throughout that $\eta$, $F$ and $\mathcal{R}$ satisfy Assumptions
\ref{H1}, \ref{H2}, \ref{H3}, respectively, and we impose the following
further assumptions on $H$ and $\varphi$. In order to obtain a strong
solution, we need to maintain only a time dependence for $H$.

\begin{condition}
[H4]\label{H4}The transforming term $H(\cdot,\cdot):\Omega\times\left[
0,T\right]  \rightarrow\mathbb{R}^{d\times d}$ is an $\mathcal{F}_{t}%
$-progressively measurable, symmetric matrix-valued process and there exist
constants $a_{H},b_{H},c_{H}>0$ such that $H_{\cdot}\left(  \omega\right)
=\left(  h_{i,j}\left(  \omega,\cdot\right)  \right)  _{d\times d}\in
C^{1}\left(  \left[  0,T\right]  ;\mathbb{R}^{d\times d}\right)  $ with the
operatorial norm satisfying%
\[
\Big\Vert\dfrac{dH_{r}\left(  \omega\right)  }{dr}\Big\Vert_{op}\leq
c_{H},\qquad\forall\,r\in\left[  0,T\right]  ,
\]
for $\mathbb{P}$-a.s.. Moreover, for all $t\in\left[  0,T\right]  $ and
$u\in\mathbb{R}^{d}$, $\mathbb{P}$-$a.s.$,%
\begin{equation}
a_{H}\left\vert u\right\vert ^{2}\leq\left\langle H_{t}u,u\right\rangle \leq
b_{H}\left\vert u\right\vert ^{2}. \label{h2}%
\end{equation}

\end{condition}

\begin{condition}
[H5]\label{H5}The function $\varphi:\mathbb{R}^{d}\rightarrow(-\infty
,+\infty]$ is a proper lower semicontinuous convex function and the terminal
datum satisfies%
\[
\mathbb{E}\,\varphi\left(  \eta\right)  <+\infty.
\]

\end{condition}

\noindent Denote by $\partial\varphi$ the subdifferential operator of
$\varphi$, i.e.%
\[
\partial\varphi\left(  x\right)  :=\{\hat{x}\in\mathbb{R}^{d}:\left\langle
\hat{x},y-x\right\rangle +\varphi\left(  x\right)  \leq\varphi\left(
y\right)  ,\text{ for all }y\in\mathbb{R}^{d}\}
\]
and define $Dom\left(  \partial\varphi\right)  :=\{x\in\mathbb{R}^{d}%
:\partial\varphi\left(  x\right)  \neq\emptyset\}$. Let us use the notation
$(x,\hat{x})\in\partial\varphi$ in order to express that $x\in Dom\left(
\partial\varphi\right)  $ and $\hat{x}\in\partial\varphi\left(  x\right)  $.
Since $\varphi$ is a proper function, we can choose, and fix, an element
$\left(  u_{0},\hat{u}_{0}\right)  \in\partial\varphi$.

\begin{definition}
The vector given by the quantity $H_{t}\hat{x}$, with $\hat{x}\in
\partial\varphi\left(  x\right)  $, is called an \textit{oblique subgradient}.
\end{definition}

\begin{definition}
\label{def of sol}Given a stochastic basis $\left(  \Omega,\mathcal{F}%
,\mathbb{P},\mathbb{F}=\{\mathcal{F}_{t}\}_{t\geq0}\right)  $, we say that a
triplet $\left(  Y,M,U\right)  $ is a strong solution for the oblique
reflected BSVI (\ref{eq to study}) if $\left(  Y,M,U\right)  \in
{\mathcal{{\mathbb{D}}}}_{d}^{0}\times\mathcal{M}_{d}^{0}\times\Lambda_{d}%
^{0}$ and, $\mathbb{P}$-a.s.,%
\[
Y_{t}+%
%TCIMACRO{\dint _{t}^{T}}%
%BeginExpansion
{\displaystyle\int_{t}^{T}}
%EndExpansion
H_{r}U_{r}dr=\eta+%
%TCIMACRO{\dint _{t}^{T}}%
%BeginExpansion
{\displaystyle\int_{t}^{T}}
%EndExpansion
F\left(  r,Y_{r},\mathcal{R}_{r}(M)\right)  dr-(M_{T}-M_{t}),\text{ }\forall
t\in\left[  0,T\right]  .
\]
Moreover, for every progressively measurable stochastic process $v$,%
\[
\mathbb{E}%
%TCIMACRO{\dint _{0}^{T}}%
%BeginExpansion
{\displaystyle\int_{0}^{T}}
%EndExpansion
\left\langle v_{r}-Y_{r},U_{r}\right\rangle dr+\mathbb{E}%
%TCIMACRO{\dint _{0}^{T}}%
%BeginExpansion
{\displaystyle\int_{0}^{T}}
%EndExpansion
\varphi\left(  Y_{r}\right)  dr\leq\mathbb{E}%
%TCIMACRO{\dint _{0}^{T}}%
%BeginExpansion
{\displaystyle\int_{0}^{T}}
%EndExpansion
\varphi\left(  v_{r}\right)  dr,
\]
i.e. $U_{r}\in\partial\varphi\left(  Y_{r}\right)  ,$ $d\mathbb{P\otimes}dr$-a.e.
\end{definition}

\begin{theorem}
[Refined version of Theorem 3.5 from Bensoussan, Li, Yam
\cite{Bensoussan/Li/Yam:2018}]\label{Main result1} Assume that we situate in
the working framework described above, and let the Assumptions \ref{H1},
\ref{H2}, \ref{H3}, \ref{H4} and \ref{H5} be satisfied. For the clarity of the
presentation, assume, in addition, that $\ell\left(  t\right)  \equiv\ell$ is
a positive constant satisfying $a_{H}>\ell^{2}C_{\mathcal{R}}^{2}/2$. Then
Problem \ref{P} admits a unique solution in the sense of Definition
\ref{def of sol}. Moreover, $\left(  Y,M,U\right)  \in{\mathcal{{\mathbb{D}}}%
}_{d}^{2}\times\mathcal{M}_{d}^{2}\times\Lambda_{d}^{2}$~.
\end{theorem}

We notice that if the generator $F$ does not depend on $z$, the structural
condition $a_{H}>\ell^{2}C_{\mathcal{R}}^{2}/2$ is automatically fulfilled,
since $\ell=0$.$\smallskip$

\subsection{Proof of the main result, Theorem \ref{Main result1}}

In order to prove the statement of Theorem \ref{Main result1}, we divide it in
several steps. We first construct a sequence of approximating equations for
Eq.(\ref{eq to study}) and establish several crucial a priori estimates for
their solutions. In the second part of the proof, devoted to identifying the
limit and establishing uniqueness, we follow a standard Cauchy-sequence
approach, while taking into account the specific features arising from the
present c\`{a}dl\`{a}g framework.

\subsubsection{Approximating sequence; a priori
estimates\label{apriori estimates}}

Even the result from Theorem \ref{Main result1} takes place under Assumption
\ref{H4}, the estimates established in this subsection remain valid under a
weaker hypothesis. More precisely, one can replace here Assumption \ref{H4} by
the following weaker assumption.

\begin{condition}
[H4W]\label{H4W}The transforming term $H(\cdot,\cdot):\Omega\times\left[
0,T\right]  \rightarrow\mathbb{R}^{d\times d}$ is a progressively measurable,
symmetric matrix-valued stochastic process and there exist some constants
$a_{H},b_{H}>0$ such that, for all $u\in\mathbb{R}^{d}$,%
\[
a_{H}\left\vert u\right\vert ^{2}\leq\left\langle H_{t}u,u\right\rangle
\leq\left\Vert H_{t}\right\Vert _{op}\left\vert u\right\vert ^{2}\leq
b_{H}\left\vert u\right\vert ^{2},\quad d\mathbb{P}\otimes dt\text{-a.e.}%
\]

\end{condition}

Let $0<\varepsilon\leq1$ and consider the approximating BSDE, driven by a
martingale term:%
\begin{equation}
Y_{t}^{\varepsilon}+%
%TCIMACRO{\dint _{t}^{T}}%
%BeginExpansion
{\displaystyle\int_{t}^{T}}
%EndExpansion
H_{r}\nabla\varphi_{\varepsilon}\left(  Y_{r}^{\varepsilon}\right)  dr=\eta+%
%TCIMACRO{\dint _{t}^{T}}%
%BeginExpansion
{\displaystyle\int_{t}^{T}}
%EndExpansion
F\left(  r,Y_{r}^{\varepsilon},\mathcal{R}_{r}(M^{\varepsilon})\right)
dr-(M_{T}^{\varepsilon}-M_{t}^{\varepsilon}),\quad\forall t\in\left[
0,T\right]  , \label{approximating eq for general case}%
\end{equation}
with $\varphi_{\varepsilon}$ being the Moreau regularization of the convex
function $\varphi$ (see Lemma \ref{conv}). We have the following result, which
is, in fact, \textit{Milestone 1} of the proof of Theorem \ref{Main result1}.

\begin{proposition}
\label{Lemma with important bounds} Let Assumptions \ref{H1}, \ref{H2},
\ref{H3}, \ref{H5}, and \ref{H4W} be satisfied, and assume, as announced in
Theorem \ref{Main result1}, that $\ell\left(  t\right)  \equiv\ell$ is a
positive constant such that the following compatibility condition holds:
$a_{H}>\ell^{2}C_{\mathcal{R}}^{2}/2$. Then, the penalized equation
(\ref{approximating eq for general case}) admits a unique solution $\left(
Y^{\varepsilon},M^{\varepsilon}\right)  \in{\mathcal{{\mathbb{D}}}}_{d}%
^{2}\times\mathcal{M}_{d}^{2}$ and the following estimates hold: there exists
a positive constant $C$, independent of $\varepsilon$, such that,%
\begin{equation}
\left\{
\begin{array}
[c]{ll}%
\left(  j\right)  \quad & \mathbb{E}\sup\limits_{t\in\left[  0,T\right]
}|Y_{t}^{\varepsilon}|^{2}+\mathbb{E}\sup\limits_{r\in\left[  0,T\right]
}\left\vert M_{r}^{\varepsilon}\right\vert ^{2}\leq C,\medskip\\
\left(  jj\right)  \quad & \mathbb{E}%
%TCIMACRO{\dint _{0}^{T}}%
%BeginExpansion
{\displaystyle\int_{0}^{T}}
%EndExpansion
\left\vert \nabla\varphi_{\varepsilon}(Y_{r}^{\varepsilon})\right\vert
^{2}dr+\mathbb{E}%
%TCIMACRO{\dint _{0}^{T}}%
%BeginExpansion
{\displaystyle\int_{0}^{T}}
%EndExpansion
\left\vert \mathcal{R}_{r}(M^{\varepsilon})\right\vert ^{2}dr\leq C,\medskip\\
\left(  jjj\right)  \quad & \mathbb{E}%
%TCIMACRO{\dint _{0}^{T}}%
%BeginExpansion
{\displaystyle\int_{0}^{T}}
%EndExpansion
\left\vert Y_{r}^{\varepsilon}-J_{\varepsilon}\left(  Y_{r}^{\varepsilon
}\right)  \right\vert ^{2}dr\leq C\varepsilon^{2},\medskip\\
\left(  jv\right)  \quad & \mathbb{E}\int_{0}^{T}\Big|\varphi
\big(J_{\varepsilon}\big(Y_{r}^{\varepsilon}\big)\big)\Big|dr\leq C\text{.}%
\end{array}
\right.  \label{estimates for the approx eq.}%
\end{equation}

\end{proposition}

\begin{proof}
[Proof of Proposition \ref{Lemma with important bounds} ]\textit{Step 1.
Approximating equation.}

Let $0<\varepsilon\leq1$ and consider the approximating BSDE
(\ref{approximating eq for general case}), with $\varphi_{\varepsilon}$
denoting the Moreau regularization of the convex function $\varphi$:%
\[
\varphi_{\varepsilon}(x)=\inf\left\{  \frac{1}{2\varepsilon}|z-x|^{2}%
+\varphi(z):z\in\mathbb{R}^{d}\right\}  =\dfrac{1}{2\varepsilon}%
|x-J_{\varepsilon}x|^{2}+\varphi(J_{\varepsilon}x),
\]
where $J_{\varepsilon}x=(I+\varepsilon\partial\varphi)^{-1}(x)$ and
$\nabla\varphi_{\varepsilon}\left(  x\right)  =(x-J_{\varepsilon
}x)/\varepsilon$. Since $\nabla\varphi_{\varepsilon}$ is $\varepsilon^{-1}%
$-Lipschitz and, by Lemma \ref{conv}-$\left(  f_{1}\right)  $, $\big|\nabla
\varphi_{\varepsilon}(y)\big|\leq\big|\hat{u}_{0}\big|+\varepsilon
^{-1}\left\vert y-u_{0}\right\vert $, the mapping $\tilde{F}_{\varepsilon
}\left(  t,y,z\right)  :=F\left(  t,y,z\right)  -H_{t}\nabla\varphi
_{\varepsilon}\left(  y\right)  $ satisfies the Lipschitz and linear growth
conditions, for every $t\in\left[  0,T\right]  ,$ $y,\hat{y}\in\mathbb{R}%
^{d},$ $z,\hat{z}\in\mathbb{R}^{d\times k},$%
\[
\left\{
\begin{array}
[c]{l}%
|\tilde{F}_{\varepsilon}(t,y,z)-\tilde{F}_{\varepsilon}(t,\hat{y},\hat
{z})|\leq\left(  L\left(  t\right)  +\dfrac{b_{H}}{\varepsilon}\right)
|y-\hat{y}|+\ell|z-\hat{z}|\medskip\\
|\tilde{F}_{\varepsilon}(t,y,0)|\leq\left\vert F\left(  t,0,0\right)
\right\vert +L\left(  t\right)  \left\vert y\right\vert +\dfrac{b_{H}%
}{\varepsilon}\left\vert y\right\vert +b_{H}\Big(\big|\hat{u}_{0}%
\big|+\dfrac{\left\vert u_{0}\right\vert }{\varepsilon}\Big).
\end{array}
\right.
\]
According to Proposition \ref{lip_FR}, the approximating equation
(\ref{approximating eq for general case}) admits a unique solution
$(Y^{\varepsilon},M^{\varepsilon})\in{\mathcal{{\mathbb{D}}}}_{d}^{2}%
\times\mathcal{M}_{d}^{2}$.$\smallskip$

\textit{Step 2. Boundedness of the approximating solution. }Let $0\leq t\leq
T.$ Since $\varphi_{\varepsilon}\in C^{1}(\mathbb{R}^{d};\mathbb{R})$ is a
convex function, the subdifferential inequality (\ref{cadlag subdiff ineq})
yields%
\begin{equation}%
\begin{array}
[c]{l}%
\varphi_{\varepsilon}\left(  Y_{t}^{\varepsilon}\right)  +%
%TCIMACRO{\dint _{t}^{T}}%
%BeginExpansion
{\displaystyle\int_{t}^{T}}
%EndExpansion
\left\langle \nabla\varphi_{\varepsilon}\left(  Y_{r}^{\varepsilon}\right)
,H_{r}\nabla\varphi_{\varepsilon}\left(  Y_{r}^{\varepsilon}\right)
\right\rangle dr\medskip\\
\quad\quad\leq\varphi_{\varepsilon}\left(  \eta\right)  +%
%TCIMACRO{\dint _{t}^{T}}%
%BeginExpansion
{\displaystyle\int_{t}^{T}}
%EndExpansion
\left\langle \nabla\varphi_{\varepsilon}(Y_{r}^{\varepsilon}),F(r,Y_{r}%
^{\varepsilon},\mathcal{R}_{r}(M^{\varepsilon}))\right\rangle dr-%
%TCIMACRO{\dint _{t+}^{T}}%
%BeginExpansion
{\displaystyle\int_{t+}^{T}}
%EndExpansion
\left\langle \nabla\varphi_{\varepsilon}(Y_{r-}^{\varepsilon}),dM_{r}%
^{\varepsilon}\right\rangle .
\end{array}
\label{subdif-eps}%
\end{equation}
By Assumption \ref{H4W} (see also (\ref{h2})), we have%
\begin{equation}
a_{H}\left\vert \nabla\varphi_{\varepsilon}(Y_{r}^{\varepsilon})\right\vert
^{2}\leq\left\langle \nabla\varphi_{\varepsilon}(Y_{r}^{\varepsilon}),H\left(
r,Y_{r}\right)  \nabla\varphi_{\varepsilon}(Y_{r}^{\varepsilon})\right\rangle
. \label{e1}%
\end{equation}
Let $\lambda>0.$ By Assumption \ref{H2}-$\left(  i\right)  $ and the
elementary inequality
\[
x\left(  y+u+v\right)  \leq\lambda x^{2}+\dfrac{1}{4\lambda}\left(
y+u+v\right)  ^{2}\leq\lambda x^{2}+\dfrac{1}{2\lambda}\left[  \left(
y+u\right)  ^{2}+v^{2}\right]  ,
\]
for all $x,y,u,v\geq0$, we deduce%
\begin{equation}%
\begin{array}
[c]{l}%
\left\langle \nabla\varphi_{\varepsilon}(Y_{r}^{\varepsilon}),F(r,Y_{r}%
^{\varepsilon},\mathcal{R}_{r}(M^{\varepsilon}))\right\rangle \leq\left\vert
\nabla\varphi_{\varepsilon}(Y_{r}^{\varepsilon})\right\vert \left[  \left\vert
F(r,0,0)\right\vert +L\left(  r\right)  \left\vert Y_{r}^{\varepsilon
}\right\vert +\ell\left\vert \mathcal{R}(M^{\varepsilon})_{r}\right\vert
\right]  \medskip\\
\quad\quad\leq\lambda\left\vert \nabla\varphi_{\varepsilon}(Y_{r}%
^{\varepsilon})\right\vert ^{2}+\dfrac{1}{2\lambda}\left[  2\left\vert
F(r,0,0)\right\vert ^{2}+2L^{2}\left(  r\right)  \left\vert Y_{r}%
^{\varepsilon}\right\vert ^{2}+\ell^{2}\left\vert \mathcal{R}_{r}%
(M^{\varepsilon})\right\vert ^{2}\right]
\end{array}
\label{e2}%
\end{equation}
Insert the two inequalities (\ref{e1}) and (\ref{e2}) into (\ref{subdif-eps})
and take into account that, for all $y\in\mathbb{R}^{d}$, and a fixed
arbitrary $\left(  u_{0},\hat{u}_{0}\right)  \in\partial\varphi,$ we have%
\[
\varphi\left(  y\right)  \geq\varphi_{\varepsilon}\left(  y\right)
\geq\varphi\left(  J_{\varepsilon}y\right)  \geq\varphi\left(  u_{0}\right)
-\left\vert \hat{u}_{0}\right\vert \left\vert y-u_{0}\right\vert
-\varepsilon\left\vert \hat{u}_{0}\right\vert ^{2}\geq-\dfrac{1}{2}\left\vert
y\right\vert ^{2}-C_{0},
\]
with $C_{0}=\left\vert \varphi\left(  u_{0}\right)  \right\vert +\frac{1}%
{2}\left\vert \hat{u}_{0}\right\vert ^{2}+\left\vert \hat{u}_{0}\right\vert
\left\vert u_{0}\right\vert +\left\vert \hat{u}_{0}\right\vert ^{2}$. This
yields%
\begin{equation}%
\begin{array}
[c]{l}%
\left(  -\dfrac{1}{2}\big|Y_{t}^{\varepsilon}\big|^{2}-C_{0}\right)  +\left(
a_{H}-\lambda\right)
%TCIMACRO{\dint _{t}^{T}}%
%BeginExpansion
{\displaystyle\int_{t}^{T}}
%EndExpansion
\left\vert \nabla\varphi_{\varepsilon}(Y_{r}^{\varepsilon})\right\vert
^{2}dr\medskip\\
\quad\leq\varphi\left(  \eta\right)  -%
%TCIMACRO{\dint _{t+}^{T}}%
%BeginExpansion
{\displaystyle\int_{t+}^{T}}
%EndExpansion
\left\langle \nabla\varphi_{\varepsilon}(Y_{r-}^{\varepsilon}),dM_{r}%
^{\varepsilon}\right\rangle +\dfrac{1}{2\lambda}%
%TCIMACRO{\dint _{t}^{T}}%
%BeginExpansion
{\displaystyle\int_{t}^{T}}
%EndExpansion
\left[  2\left\vert F(r,0,0)\right\vert ^{2}+2L^{2}\left(  r\right)
\left\vert Y_{r}^{\varepsilon}\right\vert ^{2}+\ell^{2}\left\vert
\mathcal{R}_{r}(M^{\varepsilon})\right\vert ^{2}\right]  dr.
\end{array}
\label{ineq-1}%
\end{equation}
By the energy equality (\ref{ee}), it follows that%
\begin{equation}%
\begin{array}
[c]{l}%
|Y_{t}^{\varepsilon}|^{2}+2%
%TCIMACRO{\dint _{t}^{T}}%
%BeginExpansion
{\displaystyle\int_{t}^{T}}
%EndExpansion
\left\langle Y_{r}^{\varepsilon},H_{r}\nabla\varphi_{\varepsilon}\left(
Y_{r}^{\varepsilon}\right)  \right\rangle dr+[M^{\varepsilon}]_{T}%
-[M^{\varepsilon}]_{t}\medskip\\
\quad\quad=|\eta|^{2}+2%
%TCIMACRO{\dint _{t}^{T}}%
%BeginExpansion
{\displaystyle\int_{t}^{T}}
%EndExpansion
\left\langle Y_{r}^{\varepsilon},F\left(  r,Y_{r}^{\varepsilon},\mathcal{R}%
_{r}(M^{\varepsilon})\right)  \right\rangle dr-2%
%TCIMACRO{\dint _{t+}^{T}}%
%BeginExpansion
{\displaystyle\int_{t+}^{T}}
%EndExpansion
\left\langle Y_{r-}^{\varepsilon},dM_{r}^{\varepsilon}\right\rangle .
\end{array}
\label{eq energ}%
\end{equation}
For any $\gamma>0$, by Assumptions \ref{H2} and \ref{H4W}, and applying the
elementary inequality $2xy\leq\gamma x^{2}+y^{2}/\gamma$, valid for all
$x,y\geq0$, we obtain%
\[%
\begin{array}
[c]{l}%
2\left\langle Y_{r}^{\varepsilon},F(r,Y_{r}^{\varepsilon},\mathcal{R}%
_{r}(M^{\varepsilon}))-H_{r}\nabla\varphi_{\varepsilon}\left(  Y_{r}%
^{\varepsilon}\right)  \right\rangle \medskip\\
\quad\leq2\left\vert Y_{r}^{\varepsilon}\right\vert \left[  \left\vert
F\left(  r,0,0\right)  \right\vert +L\left(  r\right)  \left\vert
Y_{r}^{\varepsilon}\right\vert +\ell\left\vert \mathcal{R}_{r}(M^{\varepsilon
})\right\vert +b_{H}\left\vert \nabla\varphi_{\varepsilon}\left(
Y_{r}^{\varepsilon}\right)  \right\vert \right]  \medskip\\
\quad\leq\gamma\left\vert F\left(  r,0,0\right)  \right\vert ^{2}%
+\gamma\left\vert \mathcal{R}_{r}(M^{\varepsilon})\right\vert ^{2}+\left(
\dfrac{1}{\gamma}+2L\left(  r\right)  +\dfrac{\ell^{2}}{\gamma}+\dfrac
{b_{H}^{2}}{\gamma}\right)  \left\vert Y_{r}^{\varepsilon}\right\vert
^{2}+\gamma\left\vert \nabla\varphi_{\varepsilon}\left(  Y_{r}^{\varepsilon
}\right)  \right\vert ^{2}.
\end{array}
\]
Inserting the above inequality into (\ref{eq energ}), we obtain, for all
$0\leq t\leq T$,
\begin{equation}%
\begin{array}
[c]{ccl}%
|Y_{t}^{\varepsilon}|^{2}+\left(  [M^{\varepsilon}]_{T}-[M^{\varepsilon}%
]_{t}\right)  & \leq & |\eta|^{2}+\gamma%
%TCIMACRO{\dint _{t}^{T}}%
%BeginExpansion
{\displaystyle\int_{t}^{T}}
%EndExpansion
\left(  \left\vert F\left(  r,0,0\right)  \right\vert ^{2}+\left\vert
\mathcal{R}_{r}(M^{\varepsilon})\right\vert ^{2}+\left\vert \nabla
\varphi_{\varepsilon}\left(  Y_{r}^{\varepsilon}\right)  \right\vert
^{2}\right)  dr\medskip\\
& + &
%TCIMACRO{\dint _{t}^{T}}%
%BeginExpansion
{\displaystyle\int_{t}^{T}}
%EndExpansion
\left(  \dfrac{1}{\gamma}+2L\left(  r\right)  +\dfrac{\ell^{2}}{\gamma}%
+\dfrac{b_{H}^{2}}{\gamma}\right)  \left\vert Y_{r}^{\varepsilon}\right\vert
^{2}dr-2%
%TCIMACRO{\dint _{t+}^{T}}%
%BeginExpansion
{\displaystyle\int_{t+}^{T}}
%EndExpansion
\left\langle Y_{r-}^{\varepsilon},dM_{r}^{\varepsilon}\right\rangle .
\end{array}
\label{ineq-ee-2}%
\end{equation}
Both stochastic integrals appearing in (\ref{ineq-1}) and (\ref{ineq-ee-2})
have zero expectation. Indeed, by the Burkholder--Davis--Gundy inequality
(\ref{BDG1}) written with $p=1$, followed by the Cauchy--Schwarz inequality,%
\begin{align*}
\mathbb{E}\sup\limits_{t\in\left[  0,T\right]  }\Big|\int_{0+}^{t}%
\big\langle Y_{r-}^{\varepsilon},dM_{r}^{\varepsilon}\big\rangle\Big|  &  \leq
C\,\mathbb{E}\Big(\int_{0+}^{T}\big|Y_{r-}^{\varepsilon}\big|^{2}%
d\big[M^{\varepsilon}\big]_{r}\Big)^{1/2}\\
&  \leq C\Big(\mathbb{E}\sup\limits_{r\in\left[  0,T\right]  }\big|Y_{r}%
^{\varepsilon}\big|^{2}\Big)^{1/2}\Big(\mathbb{E}\big|M_{T}^{\varepsilon
}\big|^{2}\Big)^{1/2}<+\infty,
\end{align*}
because $\big(Y^{\varepsilon},M^{\varepsilon}\big)\in{\mathcal{{\mathbb{D}}}%
}_{d}^{2}\times\mathcal{M}_{d}^{2}$; the integral is therefore a uniformly
integrable martingale. The same computation applies to $\int_{0+}^{\cdot
}\big\langle\nabla\varphi_{\varepsilon}(Y_{r-}^{\varepsilon}),dM_{r}%
^{\varepsilon}\big\rangle$, since $\nabla\varphi_{\varepsilon}$ is
$\varepsilon^{-1}$-Lipschitz and, therefore, $\sup_{r}\big|\nabla
\varphi_{\varepsilon}(Y_{r}^{\varepsilon})\big|\leq\left\vert \hat{u}%
_{0}\right\vert +\varepsilon^{-1}\sup_{r}\big|Y_{r}^{\varepsilon}-u_{0}\big|$,
according to (Lemma \ref{conv}-$\left(  f_{1}\right)  $). Add now, side by
side, the inequalities (\ref{ineq-1}) and (\ref{ineq-ee-2}), take expectations
and use the assumption (\ref{lip_R}) involving the mapping $\mathcal{R}$. We
obtain%
\begin{equation}%
\begin{array}
[c]{l}%
\mathbb{E}|Y_{t}^{\varepsilon}|^{2}+2\left(  a_{H}-\lambda-\gamma\right)
\mathbb{E}%
%TCIMACRO{\dint _{t}^{T}}%
%BeginExpansion
{\displaystyle\int_{t}^{T}}
%EndExpansion
\left\vert \nabla\varphi_{\varepsilon}(Y_{r}^{\varepsilon})\right\vert
^{2}dr+2\left(  1-\dfrac{\ell^{2}C_{\mathcal{R}}^{2}}{2\lambda}-\gamma
C_{\mathcal{R}}^{2}\right)  \mathbb{E}\left\vert M_{T}^{\varepsilon}%
-M_{t}^{\varepsilon}\right\vert ^{2}\medskip\\
\quad\leq2C_{0}+2\mathbb{E}|\eta|^{2}+2\mathbb{E}\varphi\left(  \eta\right)
+2\left(  \dfrac{1}{\lambda}+\gamma\right)
%TCIMACRO{\dint _{t}^{T}}%
%BeginExpansion
{\displaystyle\int_{t}^{T}}
%EndExpansion
\mathbb{E}\left\vert F\left(  r,0,0\right)  \right\vert ^{2}dr+%
%TCIMACRO{\dint _{t}^{T}}%
%BeginExpansion
{\displaystyle\int_{t}^{T}}
%EndExpansion
\mathbb{E}\left\vert Y_{r}^{\varepsilon}\right\vert ^{2}dK_{r}^{\gamma
,\lambda},
\end{array}
\label{ineq to transform}%
\end{equation}
where%
\[
K_{t}^{\gamma,\lambda}:=2\int_{0}^{t}\left(  \dfrac{L^{2}\left(  r\right)
}{\lambda}+\dfrac{1}{\gamma}+2L\left(  r\right)  +\frac{\ell^{2}}{\gamma
}+\frac{b_{H}^{2}}{\gamma}\right)  dr.
\]
Fix $\lambda_{0}$ and $\gamma_{0}$ such that%
\[
\dfrac{1}{2}\ell^{2}C_{\mathcal{R}}^{2}<\lambda_{0}<a_{H}\quad\quad
\text{and}\quad\quad0<\gamma_{0}<\left(  a_{H}-\lambda_{0}\right)
\wedge\dfrac{2\lambda_{0}-\ell^{2}C_{\mathcal{R}}^{2}}{2\lambda_{0}%
C_{\mathcal{R}}^{2}}%
\]
(recall that we assumed $\ell^{2}C_{\mathcal{R}}^{2}/2<a_{H}$, so that such a
choice is possible) and put $\left(  \lambda,\gamma\right)  =\left(
\lambda_{0},\gamma_{0}\right)  $ in (\ref{ineq to transform}). Denoting
$\alpha_{0}=2\left(  a_{H}-\lambda_{0}-\gamma_{0}\right)  >0$, $\beta
_{0}=\frac{1}{\lambda_{0}}+\gamma_{0}$, $\rho_{0}=2%
%TCIMACRO{\TeXButton{big(}{\big(}}%
%BeginExpansion
\big(%
%EndExpansion
1-\frac{\ell^{2}C_{\mathcal{R}}^{2}}{2\lambda_{0}}-\gamma_{0}C_{\mathcal{R}%
}^{2}%
%TCIMACRO{\TeXButton{big)}{\big)}}%
%BeginExpansion
\big)%
%EndExpansion
>0$ and $\kappa_{t}=K_{t}^{\gamma_{0},\lambda_{0}}$, the inequality
(\ref{ineq to transform}) becomes%
\begin{equation}%
\begin{array}
[c]{l}%
\mathbb{E}|Y_{t}^{\varepsilon}|^{2}+\alpha_{0}\mathbb{E}%
%TCIMACRO{\dint _{t}^{T}}%
%BeginExpansion
{\displaystyle\int_{t}^{T}}
%EndExpansion
\left\vert \nabla\varphi_{\varepsilon}(Y_{r}^{\varepsilon})\right\vert
^{2}dr+\rho_{0}\mathbb{E}\left\vert M_{T}^{\varepsilon}-M_{t}^{\varepsilon
}\right\vert ^{2}\medskip\\
\quad\leq\left[  2C_{0}+2\mathbb{E}|\eta|^{2}+2\mathbb{E}\left\vert
\varphi\left(  \eta\right)  \right\vert +2\beta_{0}%
%TCIMACRO{\dint _{0}^{T}}%
%BeginExpansion
{\displaystyle\int_{0}^{T}}
%EndExpansion
\mathbb{E~}\left\vert F\left(  r,0,0\right)  \right\vert ^{2}dr\right]  +%
%TCIMACRO{\dint _{t}^{T}}%
%BeginExpansion
{\displaystyle\int_{t}^{T}}
%EndExpansion
\mathbb{E}\left\vert Y_{r}^{\varepsilon}\right\vert ^{2}d\kappa_{r}.
\end{array}
\label{ineq2}%
\end{equation}
Knowing that $\mathbb{E}\left\vert M_{T}^{\varepsilon}-M_{t}^{\varepsilon
}\right\vert ^{2}=\mathbb{E}\left\vert M_{T}^{\varepsilon}\right\vert
^{2}-\mathbb{E}\left\vert M_{t}^{\varepsilon}\right\vert ^{2}=%
%TCIMACRO{\dint _{t}^{T}}%
%BeginExpansion
{\displaystyle\int_{t}^{T}}
%EndExpansion
d\mathbb{E}\left\vert M_{r}^{\varepsilon}\right\vert ^{2}$, we can apply Lemma
\ref{Gronwall} to
\[
\Theta_{s}:=\mathbb{E}\left\vert Y_{s}^{\varepsilon}\right\vert ^{2}+\rho
_{0}\int_{s}^{T}d\mathbb{E}\left\vert M_{r}^{\varepsilon}\right\vert
^{2}+\alpha_{0}\int_{s}^{T}\mathbb{E}\left\vert \nabla\varphi_{\varepsilon
}(Y_{r}^{\varepsilon})\right\vert ^{2}dr,
\]
with $dV_{r}=d\kappa_{r}$. From (Lemma \ref{Gronwall} - (\ref{g4})) we have,
\[%
\begin{array}
[c]{l}%
\mathbb{E}\left\vert Y_{t}^{\varepsilon}\right\vert ^{2}+\rho_{0}\left(
\mathbb{E}\left\vert M_{T}^{\varepsilon}\right\vert ^{2}-\mathbb{E~}\left\vert
M_{t}^{\varepsilon}\right\vert ^{2}\right)  +\alpha_{0}\mathbb{E}%
%TCIMACRO{\dint _{t}^{T}}%
%BeginExpansion
{\displaystyle\int_{t}^{T}}
%EndExpansion
\left\vert \nabla\varphi_{\varepsilon}(Y_{r}^{\varepsilon})\right\vert
^{2}dr\medskip\\
\quad\leq e^{\kappa_{T}}\left[  2C_{0}+2\mathbb{E}\left\vert \eta\right\vert
^{2}+2\mathbb{E}\left\vert \varphi\left(  \eta\right)  \right\vert +2\beta
_{0}\mathbb{E}%
%TCIMACRO{\dint _{0}^{T}}%
%BeginExpansion
{\displaystyle\int_{0}^{T}}
%EndExpansion
\left\vert F\left(  r,0,0\right)  \right\vert ^{2}dr\right]  ,\quad\text{for
all }t\in\left[  0,T\right]  .
\end{array}
\]
Since, according to (\ref{bdg adapted}), $\mathbb{E}\sup_{r\in\left[
0,T\right]  }\left\vert M_{r}^{\varepsilon}\right\vert ^{2}\leq4\mathbb{E}%
\left\vert M_{T}^{\varepsilon}\right\vert ^{2}$ and $M_{0}^{\varepsilon}=0$,
we get%
\begin{equation}
\sup_{t\in\left[  0,T\right]  }\mathbb{E}\left\vert Y_{t}^{\varepsilon
}\right\vert ^{2}+\mathbb{E}\sup_{r\in\left[  0,T\right]  }\left\vert
M_{r}^{\varepsilon}\right\vert ^{2}+\mathbb{E}%
%TCIMACRO{\dint _{0}^{T}}%
%BeginExpansion
{\displaystyle\int_{0}^{T}}
%EndExpansion
\left\vert \nabla\varphi_{\varepsilon}(Y_{r}^{\varepsilon})\right\vert
^{2}dr\leq C, \label{bound1}%
\end{equation}
where $C$ denotes a positive constant independent of $\varepsilon$, whose
value may change from one line to another. Moreover, from the form of
$\nabla\varphi_{\varepsilon}$,%
\begin{equation}
\mathbb{E}%
%TCIMACRO{\dint _{0}^{T}}%
%BeginExpansion
{\displaystyle\int_{0}^{T}}
%EndExpansion
\left\vert Y_{r}^{\varepsilon}-J_{\varepsilon}\left(  Y_{r}^{\varepsilon
}\right)  \right\vert ^{2}dr=\varepsilon^{2}\mathbb{E}%
%TCIMACRO{\dint _{0}^{T}}%
%BeginExpansion
{\displaystyle\int_{0}^{T}}
%EndExpansion
\left\vert \nabla\varphi_{\varepsilon}(Y_{r}^{\varepsilon})\right\vert
^{2}dr\leq C\varepsilon^{2}. \label{bound3}%
\end{equation}
We also remark that%
\begin{equation}
\mathbb{E}%
%TCIMACRO{\dint _{0}^{T}}%
%BeginExpansion
{\displaystyle\int_{0}^{T}}
%EndExpansion
\left\vert \mathcal{R}(M^{\varepsilon})_{r}\right\vert ^{2}dr\leq
C_{\mathcal{R}}^{2}\mathbb{E}\left\vert M_{T}^{\varepsilon}\right\vert
^{2}\leq C. \label{bound4}%
\end{equation}
From the inequality (\ref{ineq-ee-2}), using (\ref{bound1}) and (\ref{bound4}%
), we infer%
\[%
\begin{array}
[c]{lcl}%
\mathbb{E}\sup\limits_{t\in\left[  0,T\right]  }|Y_{t}^{\varepsilon}|^{2} &
\leq & \mathbb{E}\left\vert \eta\right\vert ^{2}+\gamma_{0}\mathbb{E}%
%TCIMACRO{\dint _{0}^{T}}%
%BeginExpansion
{\displaystyle\int_{0}^{T}}
%EndExpansion
\left(  \left\vert F\left(  r,0,0\right)  \right\vert ^{2}+\left\vert
\mathcal{R}(M^{\varepsilon})_{r}\right\vert ^{2}+\left\vert \nabla
\varphi_{\varepsilon}\left(  Y_{r}^{\varepsilon}\right)  \right\vert
^{2}\right)  dr\medskip\\
&  & +\sup\limits_{t\in\left[  0,T\right]  }\mathbb{E}\left\vert
Y_{t}^{\varepsilon}\right\vert ^{2}%
%TCIMACRO{\dint _{0}^{T}}%
%BeginExpansion
{\displaystyle\int_{0}^{T}}
%EndExpansion
\left(  \dfrac{1}{\gamma_{0}}+2L\left(  r\right)  +\dfrac{\ell^{2}}{\gamma
_{0}}+\dfrac{b_{H}^{2}}{\gamma_{0}}\right)  dr+\mathbb{E}\sup\limits_{t\in
\left[  0,T\right]  }\left\vert
%TCIMACRO{\dint _{t+}^{T}}%
%BeginExpansion
{\displaystyle\int_{t+}^{T}}
%EndExpansion
\left\langle 2Y_{r-}^{\varepsilon},dM_{r}^{\varepsilon}\right\rangle
\right\vert \medskip\\
& \leq & C_{1}+4\mathbb{E}\sup\limits_{t\in\left[  0,T\right]  }\left\vert
%TCIMACRO{\dint _{0+}^{t}}%
%BeginExpansion
{\displaystyle\int_{0+}^{t}}
%EndExpansion
\left\langle Y_{r-}^{\varepsilon},dM_{r}^{\varepsilon}\right\rangle
\right\vert \medskip\\
& \overset{(\text{\ref{BDG1}})}{\leq} & C_{1}+C\times\mathbb{E}\left(
%TCIMACRO{\dint _{0+}^{T}}%
%BeginExpansion
{\displaystyle\int_{0+}^{T}}
%EndExpansion
\left\vert Y_{r-}^{\varepsilon}\right\vert ^{2}d\left[  M^{\varepsilon
}\right]  _{r}\right)  ^{1/2}\leq C_{1}+C\times\mathbb{E}\left(  \left[
M^{\varepsilon}\right]  _{T}^{1/2}\sup_{r\in\left[  0,T\right]  }\left\vert
Y_{r-}^{\varepsilon}\right\vert \right)  \medskip\\
& \leq & C_{1}+C_{2}\left(  \mathbb{E}\left[  M^{\varepsilon}\right]
_{T}\right)  ^{1/2}+\dfrac{1}{2}\mathbb{E}\sup\limits_{t\in\left[  0,T\right]
}\left\vert Y_{t}^{\varepsilon}\right\vert ^{2}\overset{(\ref{bound1})}{\leq
}C_{3}+\dfrac{1}{2}\mathbb{E}\sup\limits_{t\in\left[  0,T\right]  }\left\vert
Y_{t}^{\varepsilon}\right\vert ^{2},
\end{array}
\]
where $C_{1},C_{2},C_{3}$ are $\varepsilon$-independent constants. Since
$Y^{\varepsilon}\in{\mathcal{{\mathbb{D}}}}_{d}^{2}~$, the quantity
$\mathbb{E}\sup_{t\in\left[  0,T\right]  }|Y_{t}^{\varepsilon}|^{2}$ is finite
and can be absorbed into the left-hand side. Consequently,%
\begin{equation}
\mathbb{E}\sup\limits_{t\in\left[  0,T\right]  }|Y_{t}^{\varepsilon}|^{2}\leq
C. \label{bound5}%
\end{equation}

It remains to prove $\left(  jv\right)  $. By (Lemma \ref{conv}-$\left(
f_{4}\right)  $), applied at $y=Y_{r}^{\varepsilon}$ and with $\varepsilon
\leq1$,
\[
\left\vert \varphi(J_{\varepsilon}(Y_{r}^{\varepsilon}))\right\vert
\leq\left\vert \varphi\left(  u_{0}\right)  \right\vert +\left\langle
\nabla\varphi_{\varepsilon}\left(  Y_{r}^{\varepsilon}\right)  ,J_{\varepsilon
}\left(  Y_{r}^{\varepsilon}\right)  -u_{0}\right\rangle +2\left\vert \hat
{u}_{0}\right\vert \left\vert Y_{r}^{\varepsilon}-u_{0}\right\vert
+2\left\vert \hat{u}_{0}\right\vert ^{2}.
\]
Integrating on $\Omega\times\left[  0,T\right]  $, the second term is bounded
by the Cauchy--Schwarz inequality,
\[
\mathbb{E}\int_{0}^{T}\left\langle \nabla\varphi_{\varepsilon}\left(
Y_{r}^{\varepsilon}\right)  ,J_{\varepsilon}\left(  Y_{r}^{\varepsilon
}\right)  -u_{0}\right\rangle dr\leq\left(  \mathbb{E}\int_{0}^{T}\left\vert
\nabla\varphi_{\varepsilon}\left(  Y_{r}^{\varepsilon}\right)  \right\vert
^{2}dr\right)  ^{1/2}\left(  \mathbb{E}\int_{0}^{T}\left\vert J_{\varepsilon
}\left(  Y_{r}^{\varepsilon}\right)  -u_{0}\right\vert ^{2}dr\right)  ^{1/2},
\]
and $\left\vert J_{\varepsilon}\left(  Y_{r}^{\varepsilon}\right)
-u_{0}\right\vert \leq\left\vert Y_{r}^{\varepsilon}-u_{0}\right\vert
+\varepsilon\left\vert \hat{u}_{0}\right\vert $, since $J_{\varepsilon}$ is
nonexpansive and $J_{\varepsilon}(u_{0}+\varepsilon\hat{u}_{0})=u_{0}$. Both
factors are bounded, uniformly with respect to $\varepsilon$, according to
$\left(  j\right)  $ and $\left(  jj\right)  $, and so is the third term. As
consequence $\left(  jv\right)  $ also holds. The proof of Proposition
\ref{Lemma with important bounds} is now complete.\hfill
\end{proof}

\subsubsection{Convergence of the approximating solutions\label{convergence}}

\begin{proof}
[Proof of Theorem \ref{Main result1}]\textit{Milestone 2: Cauchy sequence}

From now on, the proof follows classical arguments, so we only sketch the main
lines, emphasizing the features specific to the c\`{a}dl\`{a}g setting. For
further details on the technique, the reader is referred to Gassous,
R\u{a}\c{s}canu and Rotenstein \cite{Gassous/Rascanu/Rotenstein:15}.

Let $\left(  Y^{\varepsilon},M^{\varepsilon}\right)  $ and $(Y^{\delta
},M^{\delta})$ be the solutions of the approximating equation
(\ref{approximating eq for general case}) corresponding to $0<\varepsilon
\leq1$ and $0<\delta\leq1$, respectively, and set
\[
Y_{t}^{\varepsilon,\delta}:=Y_{t}^{\varepsilon}-Y_{t}^{\delta},\quad
M_{t}^{\varepsilon,\delta}:=M_{t}^{\varepsilon}-M_{t}^{\delta}\quad
\text{and}\quad U_{t}^{\varepsilon,\delta}:=\nabla\varphi_{\varepsilon}\left(
Y_{t}^{\varepsilon}\right)  -\nabla\varphi_{\delta}(Y_{t}^{\delta}).
\]
Subtracting the $\delta$-approximating equation from the $\varepsilon
$-approximating equation (\ref{approximating eq for general case}), we get,
for all $0\leq s\leq t\leq T$,
\begin{equation}
Y_{s}^{\varepsilon,\delta}+\int_{s}^{t}H_{r}U_{r}^{\varepsilon,\delta
}\,dr=Y_{t}^{\varepsilon,\delta}+\int_{s}^{t}F_{r}^{\varepsilon,\delta}\,dr-%
%TCIMACRO{\dint _{s+}^{t}}%
%BeginExpansion
{\displaystyle\int_{s+}^{t}}
%EndExpansion
dM_{r}^{\varepsilon,\delta}, \label{difference of pen eqs.}%
\end{equation}
where
\[
F_{r}^{\varepsilon,\delta}:=F\left(  r,Y_{r}^{\varepsilon},\mathcal{R}%
_{r}\left(  M^{\varepsilon}\right)  \right)  -F(r,Y_{r}^{\delta}%
,\mathcal{R}_{r}(M^{\delta})).
\]
Note that $H_{t}^{-1/2}$ is a positive definite symmetric matrix, such that
$H_{\cdot}^{-1/2}\left(  \omega\right)  \in C^{1}\left(  \left[  0,T\right]
;\mathbb{R}^{d\times d}\right)  $ and $\left\Vert H_{t}^{-1/2}\right\Vert
_{op}\leq1/\sqrt{a_{H}}.$ Consider $G_{r}:=dH_{r}^{-1/2}/dr$, satisfying
$\left\Vert G_{r}\right\Vert _{op}\leq c:=\dfrac{c_{H}}{2a_{H}^{3/2}}$
according to Assumption \ref{H4}. It\^{o}'s formula gives, for all $0\leq
s\leq t\leq T$,%
\[%
\begin{array}
[c]{l}%
H_{s}^{-1/2}Y_{s}^{\varepsilon,\delta}+%
%TCIMACRO{\dint _{s}^{t}}%
%BeginExpansion
{\displaystyle\int_{s}^{t}}
%EndExpansion
G_{r}Y_{r}^{\varepsilon,\delta}dr+%
%TCIMACRO{\dint _{s}^{t}}%
%BeginExpansion
{\displaystyle\int_{s}^{t}}
%EndExpansion
H_{r}^{-1/2}H_{r}U_{r}^{\varepsilon,\delta}dr\medskip\\
\quad\quad=H_{t}^{-1/2}Y_{t}^{\varepsilon,\delta}+%
%TCIMACRO{\dint _{s}^{t}}%
%BeginExpansion
{\displaystyle\int_{s}^{t}}
%EndExpansion
H_{r}^{-1/2}F_{r}^{\varepsilon,\delta}\,dr-%
%TCIMACRO{\dint _{s+}^{t}}%
%BeginExpansion
{\displaystyle\int_{s+}^{t}}
%EndExpansion
H_{r}^{-1/2}dM_{r}^{\varepsilon,\delta}.
\end{array}
\]
Apply the Energy Equality (see (\ref{ee}) and (\ref{ic})) we have%
\begin{equation}%
\begin{array}
[c]{l}%
\left\vert H_{s}^{-1/2}Y_{s}^{\varepsilon,\delta}\right\vert ^{2}+2%
%TCIMACRO{\dint _{s}^{t}}%
%BeginExpansion
{\displaystyle\int_{s}^{t}}
%EndExpansion
\left\langle H_{r}^{-1/2}Y_{r}^{\varepsilon,\delta},G_{r}Y_{r}^{\varepsilon
,\delta}+H_{r}^{-1/2}H_{r}U_{r}^{\varepsilon,\delta}\right\rangle dr+%
%TCIMACRO{\dint _{s}^{t}}%
%BeginExpansion
{\displaystyle\int_{s}^{t}}
%EndExpansion
\mathbf{Tr}\left(  H_{r}^{-1}d[M^{\varepsilon,\delta}]_{r}\right)  \medskip\\
\quad=\left\vert H_{t}^{-1/2}Y_{t}^{\varepsilon,\delta}\right\vert ^{2}+2%
%TCIMACRO{\dint _{s}^{t}}%
%BeginExpansion
{\displaystyle\int_{s}^{t}}
%EndExpansion
\left\langle H_{r}^{-1/2}Y_{r}^{\varepsilon,\delta},H_{r}^{-1/2}%
F_{r}^{\varepsilon,\delta}\right\rangle dr-2%
%TCIMACRO{\dint _{s+}^{t}}%
%BeginExpansion
{\displaystyle\int_{s+}^{t}}
%EndExpansion
\left\langle H_{r}^{-1/2}Y_{r-}^{\varepsilon,\delta},H_{r}^{-1/2}%
dM_{r}^{\varepsilon,\delta}\right\rangle ,
\end{array}
\label{ee-eps-delta}%
\end{equation}
Regarding the terms from (\ref{ee-eps-delta}) the following estimates hold:

\begin{itemize}
\item $\mathbf{Tr}(H_{r}^{-1})d[M^{\varepsilon,\delta}]_{r}\overset
{(\ref{trH_1})}{\geq}\dfrac{1}{b_{H}}d[M^{\varepsilon,\delta}]_{r}~;$

\item $\left\vert Y_{s}^{\varepsilon,\delta}\right\vert ^{2}=\left\vert
H_{s}^{1/2}H_{s}^{-1/2}Y_{s}^{\varepsilon,\delta}\right\vert ^{2}\leq
b_{H}\left\vert H_{s}^{-1/2}Y_{s}^{\varepsilon,\delta}\right\vert ^{2};$

\item $2\left\vert \left\langle H_{r}^{-1/2}Y_{r}^{\varepsilon,\delta}%
,G_{r}Y_{r}^{\varepsilon,\delta}\right\rangle \right\vert \leq2\left\Vert
H_{r}^{-1/2}\right\Vert _{op}\left\Vert G_{r}\right\Vert _{op}\left\vert
Y_{r}^{\varepsilon,\delta}\right\vert ^{2}\leq\dfrac{2c}{\sqrt{a_{H}}%
}\left\vert Y_{r}^{\varepsilon,\delta}\right\vert ^{2}~;$

\item $2\left\vert \left\langle H_{r}^{-1/2}Y_{r}^{\varepsilon,\delta}%
,H_{r}^{-1/2}F_{r}^{\varepsilon,\delta}\right\rangle \right\vert \leq\dfrac
{2}{a_{H}}\left\vert Y_{r}^{\varepsilon,\delta}\right\vert \left(  L\left(
r\right)  \left\vert Y_{r}^{\varepsilon,\delta}\right\vert +\ell\left\vert
\mathcal{R}_{r}\left(  M^{\varepsilon}\right)  -\mathcal{R}_{r}\left(
M^{\delta}\right)  \right\vert \right)  $

$\quad\leq\left(  \dfrac{2}{a_{H}}L\left(  r\right)  +\gamma\dfrac{4\ell^{2}%
}{a_{H}^{2}}\right)  \left\vert Y_{r}^{\varepsilon,\delta}\right\vert
^{2}+\dfrac{1}{\gamma}\left\vert \mathcal{R}_{r}\left(  M^{\varepsilon
}\right)  -\mathcal{R}_{r}\left(  M^{\delta}\right)  \right\vert ^{2}~$, for
every $\gamma>0$.

\item Also, using the property
\[
\left\langle \nabla\varphi_{\varepsilon}\left(  x\right)  -\nabla
\varphi_{\delta}\left(  y\right)  ,x-y\right\rangle \geq-\left(
\varepsilon+\delta\right)  \left\vert \nabla\varphi_{\varepsilon}\left(
x\right)  \right\vert \left\vert \nabla\varphi_{\delta}\left(  y\right)
\right\vert ,\quad\text{for all }x,y,
\]
we infer

$2\left\langle H_{r}^{-1/2}Y_{r}^{\varepsilon,\delta},H_{r}^{-1/2}H_{r}%
U_{r}^{\varepsilon,\delta}\right\rangle =2\left\langle Y_{r}^{\varepsilon
,\delta},U_{r}^{\varepsilon,\delta}\right\rangle \geq-2\left(  \varepsilon
+\delta\right)  \left\vert \nabla\varphi_{\varepsilon}\left(  Y_{r}%
^{\varepsilon}\right)  \right\vert \left\vert \nabla\varphi_{\delta}\left(
Y_{r}^{\delta}\right)  \right\vert $

$\geq-\left(  \varepsilon+\delta\right)  \left[  \left\vert \nabla
\varphi_{\varepsilon}\left(  Y_{r}^{\varepsilon}\right)  \right\vert
^{2}+\left\vert \nabla\varphi_{\delta}\left(  Y_{r}^{\delta}\right)
\right\vert ^{2}\right]  .$
\end{itemize}

\noindent Plug these estimates into (\ref{ee-eps-delta}), with $\gamma
=2b_{H}C_{\mathcal{R}}^{2}$, take $t=T$ (so that $Y_{T}^{\varepsilon,\delta
}=\eta-\eta=0$) and we obtain, for all $s\in\left[  0,T\right]  $,%
\begin{equation}%
\begin{array}
[c]{l}%
\dfrac{1}{b_{H}}\left\vert Y_{s}^{\varepsilon,\delta}\right\vert ^{2}%
+{\dfrac{1}{b_{H}}\int_{s+}^{T}}d[M^{\varepsilon,\delta}]_{r}\leq\left(
\varepsilon+\delta\right)  {%
%TCIMACRO{\dint _{s}^{T}}%
%BeginExpansion
{\displaystyle\int_{s}^{T}}
%EndExpansion
}\left[  \left\vert \nabla\varphi_{\varepsilon}\left(  Y_{r}^{\varepsilon
}\right)  \right\vert ^{2}+\left\vert \nabla\varphi_{\delta}\left(
Y_{r}^{\delta}\right)  \right\vert ^{2}\right]  dr\medskip\\
\quad+{%
%TCIMACRO{\dint _{s}^{T}}%
%BeginExpansion
{\displaystyle\int_{s}^{T}}
%EndExpansion
}\beta\left(  r\right)  \left\vert Y_{r}^{\varepsilon,\delta}\right\vert
^{2}dr+\dfrac{1}{2b_{H}C_{\mathcal{R}}^{2}}{%
%TCIMACRO{\dint _{s}^{T}}%
%BeginExpansion
{\displaystyle\int_{s}^{T}}
%EndExpansion
}\left\vert \mathcal{R}_{r}\left(  M^{\varepsilon}\right)  -\mathcal{R}%
_{r}\left(  M^{\delta}\right)  \right\vert ^{2}dr-2{%
%TCIMACRO{\dint _{s+}^{T}}%
%BeginExpansion
{\displaystyle\int_{s+}^{T}}
%EndExpansion
}\left\langle H_{r}^{-1}Y_{r-}^{\varepsilon,\delta},dM_{r}^{\varepsilon
,\delta}\right\rangle ,
\end{array}
\label{ee-eps-delta1}%
\end{equation}
where $\beta\left(  r\right)  =C\left(  1+L\left(  r\right)  \right)  $, with
$C=C\left(  a_{H},b_{H},c,\ell,C_{\mathcal{R}}\right)  >0$. Apply now
(\ref{BDG1}) with $p=1$, $\left\vert \left\vert H_{r}^{-1}\right\vert
\right\vert _{op}\leq a_{H}^{-1}$ and the Cauchy--Schwarz inequality, in order
to get%
\[
\mathbb{E}\sup\limits_{t\in\left[  0,T\right]  }\left\vert \int_{0+}%
^{t}\left\langle H_{r}^{-1}Y_{r-}^{\varepsilon,\delta},dM_{r}^{\varepsilon
,\delta}\right\rangle \right\vert \leq\dfrac{C}{a_{H}}\left(  \mathbb{E}%
\sup\limits_{r\in\left[  0,T\right]  }\left\vert Y_{r}^{\varepsilon,\delta
}\right\vert ^{2}\right)  ^{1/2}\left(  \mathbb{E}\left\vert M_{T}%
^{\varepsilon,\delta}\right\vert ^{2}\right)  ^{1/2}<+\infty.
\]
So, $\int_{0+}^{\cdot}\left\langle H_{r}^{-1}Y_{r-}^{\varepsilon,\delta
},dM_{r}^{\varepsilon,\delta}\right\rangle $ is a uniformly integrable
martingale and its expectation vanishes. Using (\ref{ee-eps-delta1}),
(\ref{bound1}) and Assumption \ref{H3}$,$ we deduce, for all $s\in\left[
0,T\right]  $,%
\begin{equation}
\dfrac{1}{b_{H}}\mathbb{E}\left\vert Y_{s}^{\varepsilon,\delta}\right\vert
^{2}+\dfrac{1}{2b_{H}}\left(  \mathbb{E}\left\vert M_{T}^{\varepsilon}%
-M_{T}^{\delta}\right\vert ^{2}-\mathbb{E}\left\vert M_{s}^{\varepsilon}%
-M_{s}^{\delta}\right\vert ^{2}\right)  \leq C\left(  \varepsilon
+\delta\right)  +{%
%TCIMACRO{\dint _{s}^{T}}%
%BeginExpansion
{\displaystyle\int_{s}^{T}}
%EndExpansion
}\beta\left(  r\right)  \mathbb{E}\left\vert Y_{r}^{\varepsilon,\delta
}\right\vert ^{2}dr, \label{ee-eps-delta3}%
\end{equation}
which successively yields (first via the Gronwall inequality from Lemma
\ref{Gronwall})
\begin{equation}
\left\{
\begin{array}
[c]{l}%
\mathbb{E}\left\vert Y_{s}^{\varepsilon,\delta}\right\vert ^{2}\leq C\left(
\varepsilon+\delta\right)  \exp\left(  {%
%TCIMACRO{\dint _{0}^{T}}%
%BeginExpansion
{\displaystyle\int_{0}^{T}}
%EndExpansion
}\beta\left(  r\right)  dr\right)  =C^{\prime}\left(  \varepsilon
+\delta\right)  ,\medskip\\
\mathbb{E}\left\vert M_{T}^{\varepsilon}-M_{T}^{\delta}\right\vert ^{2}\leq
C\left(  \varepsilon+\delta\right)  \quad\quad\text{(setting }s=0\text{ in
(\ref{ee-eps-delta3})).}%
\end{array}
\right.  \label{ee-eps-delta2}%
\end{equation}
Now, by the Burkholder--Davis--Gundy inequality (\ref{BDG1}) for the
stochastic integral from (\ref{ee-eps-delta1})%
\[%
\begin{array}
[c]{l}%
2\mathbb{E}\sup\limits_{t\in\left[  0,T\right]  }\left\vert
%TCIMACRO{\dint _{t+}^{T}}%
%BeginExpansion
{\displaystyle\int_{t+}^{T}}
%EndExpansion
\left\langle H_{r}^{-1}Y_{r-}^{\varepsilon,\delta},dM_{r}^{\varepsilon,\delta
}\right\rangle \right\vert =2\mathbb{E}\sup\limits_{t\in\left[  0,T\right]
}\left\vert
%TCIMACRO{\dint _{0+}^{T}}%
%BeginExpansion
{\displaystyle\int_{0+}^{T}}
%EndExpansion
\left\langle H_{r}^{-1}Y_{r-}^{\varepsilon,\delta},dM_{r}^{\varepsilon,\delta
}\right\rangle -%
%TCIMACRO{\dint _{0+}^{t}}%
%BeginExpansion
{\displaystyle\int_{0+}^{t}}
%EndExpansion
\left\langle H_{r}^{-1}Y_{r-}^{\varepsilon,\delta},dM_{r}^{\varepsilon,\delta
}\right\rangle \right\vert \medskip\\
\quad\quad\quad\quad\quad\leq4\mathbb{E}\sup\limits_{t\in\left[  0,T\right]
}\left\vert
%TCIMACRO{\dint _{0+}^{t}}%
%BeginExpansion
{\displaystyle\int_{0+}^{t}}
%EndExpansion
\left\langle H_{r}^{-1}Y_{r-}^{\varepsilon,\delta},dM_{r}^{\varepsilon,\delta
}\right\rangle \right\vert \leq C\mathbb{E}\left(
%TCIMACRO{\dint _{0+}^{T}}%
%BeginExpansion
{\displaystyle\int_{0+}^{T}}
%EndExpansion
\left\vert H_{r}^{-1}Y_{r-}^{\varepsilon,\delta}\right\vert ^{2}d\left[
M^{\varepsilon,\delta}\right]  _{r}\right)  ^{1/2}\medskip\\
\quad\quad\quad\quad\quad\leq C^{\prime}\mathbb{E}\sup\limits_{r\in\left[
0,T\right]  }\left\vert Y_{r-}^{\varepsilon,\delta}\right\vert \left(
%TCIMACRO{\dint _{0+}^{T}}%
%BeginExpansion
{\displaystyle\int_{0+}^{T}}
%EndExpansion
d[M^{\varepsilon,\delta}]_{r}\right)  ^{1/2}\leq\dfrac{1}{2}\mathbb{E}%
\sup\limits_{r\in\left[  0,T\right]  }\left\vert Y_{r}^{\varepsilon,\delta
}\right\vert ^{2}+C^{\prime\prime}\left(  \varepsilon+\delta\right)  .
\end{array}
\]
Using this last estimate, together with (\ref{ee-eps-delta1}), (\ref{bound1})
and (\ref{ee-eps-delta2}), we deduce%
\[
\mathbb{E}\sup\limits_{t\in\left[  0,T\right]  }\left\vert Y_{t}%
^{\varepsilon,\delta}\right\vert ^{2}\leq\mathbb{E}\sup\limits_{t\in\left[
0,T\right]  }\left[  \left\vert Y_{t}^{\varepsilon,\delta}\right\vert ^{2}+%
%TCIMACRO{\dint _{t}^{T}}%
%BeginExpansion
{\displaystyle\int_{t}^{T}}
%EndExpansion
d[M^{\varepsilon,\delta}]_{r}\right]  \leq\dfrac{1}{2}\mathbb{E}%
\sup\limits_{r\in\left[  0,T\right]  }\left\vert Y_{r}^{\varepsilon,\delta
}\right\vert ^{2}+C^{\prime}\left(  \varepsilon+\delta\right)
\]
Hence
\begin{equation}
\mathbb{E}\sup\limits_{t\in\left[  0,T\right]  }\left\vert Y_{t}^{\varepsilon
}-Y_{t}^{\delta}\right\vert ^{2}=\mathbb{E}\sup\limits_{t\in\left[
0,T\right]  }\left\vert Y_{t}^{\varepsilon,\delta}\right\vert ^{2}\leq
C\left(  \varepsilon+\delta\right)  . \label{eps-delta3}%
\end{equation}

\textit{Milestone 3: Convergences and identification of the limit.}%
$\smallskip$

We conclude, by (\ref{ee-eps-delta2}) and (\ref{eps-delta3}), that $\left(
Y^{\varepsilon},M^{\varepsilon}\right)  _{\varepsilon}$ is a Cauchy sequence
in ${\mathcal{{\mathbb{D}}}}_{d}^{2}\times\mathcal{M}_{d}^{2}.$ Hence there
exists $\left(  Y,M\right)  \in{\mathcal{{\mathbb{D}}}}_{d}^{2}\times$
$\mathcal{M}_{d}^{2}$ such that%
\[
\lim\limits_{\varepsilon\rightarrow0_{+}}\left[  \mathbb{E}\sup\limits_{t\in
\left[  0,T\right]  }\left\vert Y_{t}^{\varepsilon}-Y_{t}\right\vert
^{2}+\mathbb{E}\left\vert M_{T}^{\varepsilon}-M_{T}\right\vert ^{2}\right]
=0
\]
and by Lemma \ref{L2_F},%
\[
\lim\limits_{\varepsilon\rightarrow0_{+}}\mathbb{E}\sup\limits_{t\in\left[
0,T\right]  }\left\vert
%TCIMACRO{\dint _{0}^{t}}%
%BeginExpansion
{\displaystyle\int_{0}^{t}}
%EndExpansion
F(s,Y_{s}^{\varepsilon},\mathcal{R}_{s}(M^{\varepsilon}))ds-\int_{0}%
^{t}F(s,Y_{s},\mathcal{R}_{s}(M))ds\right\vert ^{2}=0.
\]
From (\ref{estimates for the approx eq.}), the family $\left(  \nabla
\varphi_{\varepsilon}(Y^{\varepsilon})\right)  _{\varepsilon}$ is bounded in
$\Lambda_{d}^{2}~$, so there exist a sequence $\varepsilon_{n}\rightarrow
0_{+}$ and $U\in\Lambda_{d}^{2}$ such that $U^{\varepsilon_{n}}=\nabla
\varphi_{\varepsilon_{n}}\left(  Y^{\varepsilon_{n}}\right)  \rightharpoonup
U$, weakly in $\Lambda_{d}^{2}$. Passing to the limit in the approximating
equation (\ref{approximating eq for general case}), we infer that%
\[
Y_{t}+%
%TCIMACRO{\dint _{t}^{T}}%
%BeginExpansion
{\displaystyle\int_{t}^{T}}
%EndExpansion
H_{r}U_{r}dr=\eta+%
%TCIMACRO{\dint _{t}^{T}}%
%BeginExpansion
{\displaystyle\int_{t}^{T}}
%EndExpansion
F\left(  r,Y_{r},\mathcal{R}_{r}(M)\right)  dr-(M_{T}-M_{t}),\text{ }\forall
t\in\left[  0,T\right]  ,\;\mathbb{\mathbb{P}}\text{-a.s.}%
\]
Since $U_{t}^{\varepsilon_{n}}=\nabla\varphi_{\varepsilon_{n}}(Y_{t}%
^{\varepsilon_{n}})$, i.e. $Y_{t}^{\varepsilon_{n}}-J_{\varepsilon_{n}}%
(Y_{t}^{\varepsilon_{n}})=\varepsilon_{n}U_{t}^{\varepsilon_{n}}$ and
$\varepsilon_{n}\Vert U^{\varepsilon_{n}}\Vert_{\Lambda_{d}^{2}}\rightarrow0$,
as $\varepsilon_{n}\rightarrow0,$ we obtain $J_{\varepsilon_{n}}%
(Y^{\varepsilon_{n}})\rightarrow Y$, strongly in $\Lambda_{d}^{2}~$. Moreover,
after passing to a further subsequence, still denoted by $\varepsilon_{n}$,
$J_{\varepsilon_{n}}(Y_{t}^{\varepsilon_{n}})\rightarrow Y_{t}$,
$d\mathbb{P\otimes}dt$-a.e.. From $U_{t}^{\varepsilon_{n}}=\nabla
\varphi_{\varepsilon_{n}}(Y_{t}^{\varepsilon_{n}})\in\partial\varphi\left(
J_{\varepsilon_{n}}(Y_{t}^{\varepsilon_{n}})\right)  $, it follows that%
\[
\left\langle U_{t}^{\varepsilon_{n}},x-Y_{t}^{\varepsilon_{n}}\right\rangle
+\varphi\left(  J_{\varepsilon_{n}}(Y_{t}^{\varepsilon_{n}})\right)
\leq\varphi(x),\qquad\forall x\in\mathbb{R}^{d}.
\]
The inequality also holds with $x$ replaced by $V_{t}$, for every $V\in
\Lambda_{d}^{2}~$; By integration on $\Omega\times\left[  0,T\right]  $,%
\[%
\begin{array}
[c]{l}%
\mathbb{E}%
%TCIMACRO{\dint _{0}^{T}}%
%BeginExpansion
{\displaystyle\int_{0}^{T}}
%EndExpansion
\left\langle U_{r}^{\varepsilon_{n}},V_{r}-Y_{r}^{\varepsilon_{n}%
}\right\rangle dr+\mathbb{E}%
%TCIMACRO{\dint _{0}^{T}}%
%BeginExpansion
{\displaystyle\int_{0}^{T}}
%EndExpansion
\left[  \varphi\left(  J_{\varepsilon_{n}}(Y_{r}^{\varepsilon_{n}})\right)
+\left\vert J_{\varepsilon_{n}}(Y_{r}^{\varepsilon_{n}})\right\vert ^{2}%
+C_{0}\right]  dr\medskip\\
\quad\quad\quad\leq\mathbb{E}%
%TCIMACRO{\dint _{0}^{T}}%
%BeginExpansion
{\displaystyle\int_{0}^{T}}
%EndExpansion
\left[  \varphi\left(  V_{r}\right)  +\left\vert J_{\varepsilon_{n}}%
(Y_{r}^{\varepsilon_{n}})\right\vert ^{2}+C_{0}\right]  dr.
\end{array}
\]
The first term converges as a weak--strong pairing, since $U^{\varepsilon_{n}%
}\rightharpoonup U$, weakly in $\Lambda_{d}^{2}$ and $V-Y^{\varepsilon_{n}%
}\rightarrow V-Y$, strongly in $\Lambda_{d}^{2}$. For the second one,
$J_{\varepsilon_{n}}(Y^{\varepsilon_{n}})\rightarrow Y$, in $\Lambda_{d}^{2}$,
hence $d\mathbb{P}\otimes dr$-a.e. along a subsequence, and Fatou's lemma
together with the lower semicontinuity of $y\longmapsto\varphi\left(
y\right)  +\left\vert y\right\vert ^{2}+C_{0}\geq0$ applies. Indeed, for the
positivity of the application to be verified, the previous coercivity assures
$\varphi\left(  y\right)  \geq\frac{1}{2}\left\vert y\right\vert ^{2}-C_{0}$,
which gives $\varphi\left(  y\right)  +\left\vert y\right\vert ^{2}+C_{0}%
\geq\frac{3}{2}\left\vert y\right\vert ^{2}\geq0$. Therefore%
\begin{equation}
\mathbb{E}\int_{0}^{T}\left\langle U_{r},V_{r}-Y_{r}\right\rangle
dr+\mathbb{E}\int_{0}^{T}\varphi\left(  Y_{r}\right)  dr\leq\mathbb{E}\int
_{0}^{T}\varphi\left(  V_{r}\right)  dr,\qquad\forall\,V\in\Lambda_{d}^{2}.
\label{subdif-limit}%
\end{equation}
The inequality (\ref{subdif-limit}) means, in fact, that $U\in\partial
\Phi\left(  Y\right)  $, where
\[
\Phi\left(  V\right)  :=\mathbb{E}\int_{0}^{T}\varphi\left(  V_{r}\right)  dr
\]
is a proper convex lower semicontinuous functional on the Hilbert space
$\Lambda_{d}^{2}$ (it is proper by $\left(  jv\right)  $ and Fatou's lemma).
Since the subdifferential of such an functional integral is given by
$\partial\Phi\left(  V\right)  =\{U\in\Lambda_{d}^{2}:U_{r}\in\partial
\varphi\left(  V_{r}\right)  ,\ d\mathbb{P}\otimes dr\text{-a.e.}\}$ (see
Br\'{e}zis \cite{Brezis:73}), we conclude that $U_{r}\in\partial\varphi\left(
Y_{r}\right)  $, $d\mathbb{P\otimes}dr$-a.e. Hence $\left(  Y,M,U\right)
\in{\mathcal{{\mathbb{D}}}}_{d}^{2}\times\mathcal{M}_{d}^{2}\times\Lambda
_{d}^{2}$ is a solution of \ref{P}.\medskip

\noindent\textit{Milestone 4: Uniqueness of the strong solution.}$\smallskip$

Let $\left(  Y,M,U\right)  $ and $(\tilde{Y},\tilde{M},\tilde{U})$ be two
strong solutions of the original equation. Subtracting the equations satisfied
by the two solutions, we obtain%
\begin{align*}
(Y_{t}-\tilde{Y}_{t})+\int_{t}^{T}H_{s}(U_{s}-\tilde{U}_{s})ds  &  =\int
_{t}^{T}\left[  F\left(  s,Y_{s},\mathcal{R}_{s}(M)\right)  -F(s,\tilde{Y}%
_{s},\mathcal{R}_{s}(\tilde{M}))\right]  ds\\
&  -\left[  (M_{T}-\tilde{M}_{T})-(M_{t}-\tilde{M}_{t})\right]  .
\end{align*}
Repeating the argument of \textit{Step 1} and applying It\^{o}'s formula to
$\left\vert H_{t}^{-1/2}(Y_{t}-\tilde{Y}_{t})\right\vert ^{2}$ - the term
$2\langle Y_{s}-\tilde{Y}_{s},U_{s}-\tilde{U}_{s}\rangle$ being now
nonnegative, by the monotonicity of $\partial\varphi$ - we obtain, in exactly
the same manner,%
\[
\mathbb{E}\sup_{t\in\lbrack0,T]}\left\vert Y_{t}-\tilde{Y}_{t}\right\vert
^{2}+\mathbb{E}\left\vert M_{T}-\tilde{M}_{T}\right\vert ^{2}=0.
\]
We obtained the uniqueness of $U$ as well. This completes the proof of Theorem
\ref{Main result1}.\hfill
\end{proof}

\section{Annex}

\subsection{Tools for the c\`{a}dl\`{a}g calculus\label{cadlag tools}}

We recall here only the arguments and tools we shall use in our c\`{a}dl\`{a}g
study. They are needed when we use: It\^{o}'s formula (\ref{ee}), the
Burkholder--Davis--Gundy inequality (\ref{BDG1}), which is the only route to
pass from $\sup_{t}\mathbb{E}$ to $\mathbb{E}\sup_{t}$~, and in the two-sided
estimate (\ref{ic-est}), used in the proof of Theorem \ref{Main result1}, for
comparing $\int H^{-1/2}dM$ with $M$.

For a process $X$, denote $X_{t-}:=\lim_{s\uparrow t}X_{s}.$ According to
Protter \cite[Chapter II, Th. 21, page 64]{Protter:05}, if $\left\{
X_{t}:t\geq0\right\}  $ is an $\mathbb{R}^{d}$-valued adapted c\`{a}dl\`{a}g
semimartingale, the stochastic integral with respect to $X$,%
\[
I_{X}\left(  Y\right)  _{t}:=\int_{0+}^{t}Y_{s-}dX_{s},\quad t>0,\quad
I_{X}\left(  Y\right)  _{0}=0;
\]
is defined by linear extension from simple adapted c\`{a}gl\`{a}d stochastic
processes
\[
Y_{t}=Y_{0}\mathbf{1}_{\left\{  0\right\}  }+%
%TCIMACRO{\dsum _{k=0}^{n-1}}%
%BeginExpansion
{\displaystyle\sum_{k=0}^{n-1}}
%EndExpansion
Y_{k}\mathbf{1}_{(T_{k},T_{k+1}]}~,
\]
where $0=T_{0}\leq T_{1}\leq...\leq T_{n}<\infty$ is a finite sequence of
stopping times and the $Y_{k}$ are $\mathcal{F}_{T_{k}}$-measurable,
$\mathbb{R}^{m\times d}$-valued bounded random variables, by%
\[
I_{X}\left(  Y\right)  _{t}:=%
%TCIMACRO{\dsum _{k=0}^{n-1}}%
%BeginExpansion
{\displaystyle\sum_{k=0}^{n-1}}
%EndExpansion
Y_{k}\left(  X_{t\wedge T_{k+1}}-X_{t\wedge T_{k}}\right)  .
\]
After this, it is extended to locally bounded predictable processes $Y$, as
follows. If $Y\in\mathbb{L}_{m\times d}^{0}$ and $\mathcal{P}_{n}:0=T_{0}%
^{n}\leq T_{1}^{n}\leq...\leq T_{k_{n}}^{n}<\infty$, $n\geq1$ is an arbitrary
sequence of random partitions of stopping times such that $T_{k_{n}}%
^{n}\nearrow\infty$ and $\sup_{k\in\overline{0,k_{n}-1}}\left\vert T_{k+1}%
^{n}-T_{k}^{n}\right\vert \rightarrow0,$ a.s., as $n\rightarrow+\infty$, then
the sequence of processes%
\begin{equation}
Y_{t}^{n}:=%
%TCIMACRO{\dsum _{k=0}^{k_{n}-1}}%
%BeginExpansion
{\displaystyle\sum_{k=0}^{k_{n}-1}}
%EndExpansion
Y_{T_{k}^{n}}\left(  X_{t\wedge T_{k+1}^{n}}-X_{t\wedge T_{k}^{n}}\right)  =%
%TCIMACRO{\dsum _{k=0}^{k_{n}-1}}%
%BeginExpansion
{\displaystyle\sum_{k=0}^{k_{n}-1}}
%EndExpansion
Y_{t\wedge T_{k}^{n}}\left(  X_{t\wedge T_{k+1}^{n}}-X_{t\wedge T_{k}^{n}%
}\right)  \label{si-1}%
\end{equation}
converges in probability to $I_{X}\left(  Y\right)  _{t}~$, uniformly on
compact subsets of $\mathbb{R}_{+}$. We note that $I_{X}:\mathbb{L}_{m\times
d}^{0}\rightarrow\mathbb{D}_{d}^{0}$ is a linear continuous mapping.

Let us now remind some important properties of the stochastic integral
$I_{X}\left(  Y\right)  $. Let $X\in\mathbb{D}_{d}^{0}$ and $Y\in
\mathbb{L}_{m\times d}^{0}.$ Each of the following properties of $X$ is
inherited by $I_{X}\left(  Y\right)  :$

\begin{itemize}
\item $X$ is a semimartingale (Protter \cite[Chapter II, Th.19, page
62]{Protter:05});

\item $X$ is a bounded variation stochastic process (Protter \cite[Chapter II,
Th.17, page 61]{Protter:05});

\item $X$ is a locally square integrable local martingale (Protter
\cite[Chapter II, Th.20, page 63]{Protter:05});

\item $X$ is a local martingale (Protter \cite[Chapter III, Th.29,
page128]{Protter:05}).
\end{itemize}

Given two semimartingales $X,Y\in\mathbb{D}_{d}^{0}$, by $[X,Y]=\left(
[X,Y]_{t}\right)  _{t\geq0}$ we denote their quadratic variation, which is
defined by%
\[
\lbrack X,Y]_{t}=\left\langle X_{t},Y_{t}\right\rangle -\int_{0+}%
^{t}\left\langle X_{s-},dY_{s}\right\rangle -\int_{0+}^{t}\left\langle
Y_{s-},dX_{s}\right\rangle .
\]
We remark that
\[
\left\langle X_{0},Y_{0}\right\rangle +%
%TCIMACRO{\dsum _{k=0}^{k_{n}-1}}%
%BeginExpansion
{\displaystyle\sum_{k=0}^{k_{n}-1}}
%EndExpansion
\left\langle X_{t\wedge T_{k+1}^{n}}-X_{t\wedge T_{k}^{n}},Y_{t\wedge
T_{k+1}^{n}}-Y_{t\wedge T_{k}^{n}}\right\rangle \rightarrow\lbrack X,Y]_{t}~,
\]
the convergence taking place in probability, uniformly with respect to $t$, on
compact subsets of $\mathbb{R}_{+}$ (abbreviated \textit{ucp}); see Protter
\cite[Chapter II, Theorem 23, page 68]{Protter:05}. The quadratic variation of
$X$, namely%
\begin{equation}
\lbrack X]_{t}:=[X,X]_{t}=\left\vert X_{t}\right\vert ^{2}-2%
%TCIMACRO{\dint _{0}^{t}}%
%BeginExpansion
{\displaystyle\int_{0}^{t}}
%EndExpansion
\left\langle X_{s-},dX_{s}\right\rangle =\text{\textit{ucp-}}\lim
\limits_{n\rightarrow\infty}\left[  \left\vert X_{0}\right\vert ^{2}+%
%TCIMACRO{\dsum \limits_{k=0}^{k_{n}-1}}%
%BeginExpansion
{\displaystyle\sum\limits_{k=0}^{k_{n}-1}}
%EndExpansion
\left\vert X_{t\wedge T_{k+1}^{n}}-X_{t\wedge T_{k}^{n}}\right\vert
^{2}\right]  \label{qv}%
\end{equation}
is a c\`{a}dl\`{a}g, nondecreasing adapted process (Protter \cite[Chapter II,
Theorem 22, page 66]{Protter:05}), with
\begin{equation}
\Delta\left[  X\right]  _{t}=\left\vert \Delta X_{t}\right\vert ^{2}%
\quad\text{and}\quad\left[  X\right]  _{t}=\left\vert X_{0}\right\vert
^{2}+\left\langle X^{c}\right\rangle _{t}+\sum_{0<s\leq t}\left\vert \Delta
X_{s}\right\vert ^{2}, \label{qv-jumps}%
\end{equation}
$X^{c}$ denoting the continuous local martingale part of $X$, and%
\[
\text{if }X\in\mathcal{M}_{d}^{2}\;\text{then }\mathbb{E}[X,X]_{T}%
=\mathbb{E}\left\vert X_{T}\right\vert ^{2}.
\]

\begin{remark}
\label{rem:two-brackets}The two brackets $\left[  M\right]  $ and
$\left\langle M\right\rangle $ of a martingale $M\in\mathcal{M}_{d}^{2}$ must
be distinguished. $\left[  M\right]  $ is defined pathwise, by (\ref{qv}), for
every semimartingale, and is merely adapted; $\left\langle M\right\rangle $ is
defined for $M\in\mathcal{M}_{d}^{2}$ as a predictable process such that
$\left\vert M\right\vert ^{2}-\left\langle M\right\rangle $ is a martingale,
or equivalently $\left\langle M\right\rangle $ is the compensator of $\left[
M\right]  $, so that $\left[  M\right]  -\left\langle M\right\rangle $ is a
martingale and
\begin{equation}
\mathbb{E}\left[  M\right]  _{t}=\mathbb{E}\left\langle M\right\rangle
_{t}=\mathbb{E}\left\vert M_{t}\right\vert ^{2}. \label{three-equal}%
\end{equation}
If $M$ is continuous then $\left[  M\right]  =\left\langle M\right\rangle $.
The simplest example in which they differ is the compensated Poisson process
$\hat{N}_{t}=N_{t}-\lambda t$, where $N$ is a Poisson process of intensity
$\lambda>0$: there $%
%TCIMACRO{\TeXButton{big[}{\big[}}%
%BeginExpansion
\big[%
%EndExpansion
\hat{N}%
%TCIMACRO{\TeXButton{big]}{\big]}}%
%BeginExpansion
\big]%
%EndExpansion
_{t}=N_{t}$ is random, while $%
%TCIMACRO{\TeXButton{big<}{\big\langle}}%
%BeginExpansion
\big\langle
%EndExpansion
\hat{N}%
%TCIMACRO{\TeXButton{big>}{\big\rangle}}%
%BeginExpansion
\big\rangle
%EndExpansion
_{t}=\lambda t$ is deterministic. The hypotheses, the norms and the apriori
estimates used along this paper are all written with the pathwise bracket
$\left[  M\right]  $, which is also the object appearing in (\ref{BDG1}) and
(\ref{ic-est}). According to (\ref{three-equal}), each of them could
equivalently be stated with $\left\langle M\right\rangle $, since only
expectations of the brackets occur.
\end{remark}

We also have%
\[
\lbrack X,Y]=\frac{1}{2}([X+Y,X+Y]-[X,X]-[Y,Y]).
\]
For the bracket of stochastic integrals, if $X,Y\in\mathbb{D}_{d}^{0}$ are
semimartingales and $H,K\in\mathbb{L}_{m\times d}^{0}$~, then, denoting by
$d\left[  X,Y\right]  _{s}:=\left(  d[X^{j},Y^{l}]_{s}\right)  _{1\leq j,l\leq
d}$ the matrix-valued bracket measure,%
\begin{equation}
\left[  I_{X}\left(  H\right)  ,I_{Y}\left(  K\right)  \right]  _{t}=\int
_{0+}^{t}\mathbf{Tr}\left(  K_{s}^{\ast}H_{s}~d\left[  X,Y\right]
_{s}\right)  . \label{ic}%
\end{equation}
Since $s\longmapsto\left[  X\right]  _{s}$ is a nonnegative-definite
matrix-valued measure, we deduce, for every $H\in\mathbb{L}_{d\times d}^{0}~$,%
\begin{equation}
\lambda_{\min}\left(  H_{s}^{\ast}H_{s}\right)  d\left[  X\right]  _{s}\leq
d\left[  I_{X}\left(  H\right)  \right]  _{s}\leq\lambda_{\max}\left(
H_{s}^{\ast}H_{s}\right)  d\left[  X\right]  _{s}=\left\Vert H_{s}\right\Vert
_{op}^{2}\,d\left[  X\right]  _{s}, \label{ic-est}%
\end{equation}
where $\left[  X\right]  _{s}:=\mathbf{Tr}\left[  X,X\right]  _{s}$,
$\lambda_{\min}\left(  H_{s}^{\ast}H_{s}\right)  $ and $\lambda_{\max}\left(
H_{s}^{\ast}H_{s}\right)  $ are the minimum eigenvalue and, respectively, the
maximum eigenvalue of the matrix $H_{s}^{\ast}H_{s}$ (see Protter
\cite[Chapter II, Theorem 29, page 75]{Protter:05}).

Let us introduce the \textit{Burkholder--Davis--Gundy inequality }suited to
our working setup: for any $p\in\lbrack1,\infty)$ there exist two constants
$c_{p},C_{p}>0,$ depending only on $p,$ such that, for all local martingales
$X,$ with $X_{0}=0,$ and any stopping time $\tau$, the following inequality
holds%
\begin{equation}
c_{p}\mathbb{E}\left[  X\right]  _{\tau}^{p/2}\leq\mathbb{E}\sup_{0\leq
t\leq\tau}\left\vert X_{t}\right\vert ^{p}\leq C_{p}\mathbb{E}\left[
X\right]  _{\tau}^{p/2}. \label{BDG}%
\end{equation}
If $X$ is a continuous local martingale, then the inequality (\ref{BDG}) holds
for all $0<p<\infty$. In the case of the stochastic integral defined for
$X\in\mathbb{D}_{d}^{0}$ (being a local martingale) and $Y\in\mathbb{L}%
_{m\times d}^{0}~$, we have, for any $1\leq p<\infty$,%
\begin{equation}
\mathbb{E}\sup_{0\leq t\leq\tau}\left\vert \int_{0+}^{t}Y_{s}dX_{s}\right\vert
^{p}\leq C_{p}\mathbb{E}\left(  \int_{0+}^{\tau}\left\vert Y_{s}\right\vert
^{2}d\left[  X\right]  _{s}\right)  ^{p/2}. \label{BDG1}%
\end{equation}

It\^{o}'s formula for semimartingales becomes:

\begin{lemma}
If $Y$ is a $d$-dimensional c\`{a}dl\`{a}g semimartingale and $u\in
C^{1,2}(\left[  0,T\right]  \times\mathbb{R}^{d};\mathbb{R})$, then $u\left(
\cdot,Y_{\cdot}\right)  $ is a semimartingale and the following formula
holds:\vspace{-0.1in}%
\[%
\begin{array}
[c]{l}%
u\left(  t,Y_{t}\right) \\
\quad=u\left(  s,Y_{s}\right)  +%
%TCIMACRO{\dint _{(s,t]}}%
%BeginExpansion
{\displaystyle\int_{(s,t]}}
%EndExpansion
\dfrac{\partial u\left(  r,Y_{r}\right)  }{\partial t}dr+%
%TCIMACRO{\dsum _{i=1}^{d}}%
%BeginExpansion
{\displaystyle\sum_{i=1}^{d}}
%EndExpansion%
%TCIMACRO{\dint _{(s,t]}}%
%BeginExpansion
{\displaystyle\int_{(s,t]}}
%EndExpansion
\dfrac{\partial u\left(  r,Y_{r-}\right)  }{\partial x_{i}}dY_{r}^{i}%
+\dfrac{1}{2}%
%TCIMACRO{\dsum _{1\leq i,j\leq d}}%
%BeginExpansion
{\displaystyle\sum_{1\leq i,j\leq d}}
%EndExpansion%
%TCIMACRO{\dint _{(s,t]}}%
%BeginExpansion
{\displaystyle\int_{(s,t]}}
%EndExpansion
\dfrac{\partial^{2}u\left(  r,Y_{r-}\right)  }{\partial x_{i}\partial x_{j}%
}d[Y^{i},Y^{j}]_{r}\smallskip\\
\quad+%
%TCIMACRO{\dsum _{s<r\leq t}}%
%BeginExpansion
{\displaystyle\sum_{s<r\leq t}}
%EndExpansion
\left\{  u(r,Y_{r})-u(r,Y_{r-})-%
%TCIMACRO{\dsum _{i=1}^{d}}%
%BeginExpansion
{\displaystyle\sum_{i=1}^{d}}
%EndExpansion
\dfrac{\partial u\left(  r,Y_{r-}\right)  }{\partial x_{i}}\Delta Y_{r}%
^{i}-\dfrac{1}{2}%
%TCIMACRO{\dsum _{1\leq i,j\leq d}}%
%BeginExpansion
{\displaystyle\sum_{1\leq i,j\leq d}}
%EndExpansion
\dfrac{\partial^{2}u\left(  r,Y_{r-}\right)  }{\partial x_{i}\partial x_{j}%
}\Delta Y_{r}^{i}\Delta Y_{r}^{j}\right\}  .
\end{array}
\]
for all $0\leq s\leq t\leq T,$ $\mathbb{P}$-$a.s.$ In particular, the energy
equality reads%
\begin{equation}
\left\vert Y_{t}\right\vert ^{2}=\left\vert Y_{s}\right\vert ^{2}+2%
%TCIMACRO{\dint _{(s,t]}}%
%BeginExpansion
{\displaystyle\int_{(s,t]}}
%EndExpansion
\left\langle Y_{r-},dY_{r}\right\rangle +[Y,Y]_{t}-[Y,Y]_{s}~. \label{ee}%
\end{equation}

\end{lemma}

\noindent In the formula above, $Y_{r}=Y_{r-}$ for $dr$-a.e. $r$, so the
integrals with respect to $dr$ may equivalently be taken on $\left[
s,t\right]  $.

Arguing as in Pardoux and R\u{a}\c{s}canu \cite[Lemma 2.37]%
{Pardoux/Rascanu:14}, we obtain that, if $Y$ is a $d$-dimensional
c\`{a}dl\`{a}g semimartingale and $\psi\in C^{1}(\mathbb{R}^{d};\mathbb{R})$
is convex, then the following c\`{a}dl\`{a}g stochastic subdifferential
inequality holds, for every $0\leq t<s\leq T$:%
\begin{equation}%
%TCIMACRO{\dint _{t+}^{s}}%
%BeginExpansion
{\displaystyle\int_{t+}^{s}}
%EndExpansion
\left\langle \nabla_{y}\psi\left(  Y_{r-}\right)  ,dY_{r}\right\rangle
\leq\psi\left(  Y_{s}\right)  -\psi\left(  Y_{t}\right)  .
\label{cadlag subdiff ineq}%
\end{equation}
Indeed, it suffices to use the definition (\ref{si-1}) of the stochastic
integral along partitions starting at $T_{0}^{n}=t$, together with the
subgradient inequality%
\[%
%TCIMACRO{\dsum _{k=0}^{k_{n}-1}}%
%BeginExpansion
{\displaystyle\sum_{k=0}^{k_{n}-1}}
%EndExpansion
\nabla_{y}\psi\left(  Y_{t\wedge T_{k}^{n}}\right)  \left(  Y_{t\wedge
T_{k+1}^{n}}-Y_{t\wedge T_{k}^{n}}\right)  \leq%
%TCIMACRO{\dsum _{k=0}^{k_{n}-1}}%
%BeginExpansion
{\displaystyle\sum_{k=0}^{k_{n}-1}}
%EndExpansion
\left[  \psi\left(  Y_{t\wedge T_{k+1}^{n}}\right)  -\psi\left(  Y_{t\wedge
T_{k}^{n}}\right)  \right]  ;
\]
and to pass to the limit, which is legitimate by Protter \cite[Chapter II,
Theorem 21, page 64]{Protter:05}. Inequality (\ref{cadlag subdiff ineq}) will
prove useful in deriving estimates for the sequence of solutions of the
penalized equations without imposing strong assumptions on the multivalued term.

\subsection{On the boundedness of symmetric matrices}

We note that a symmetric matrix $H\in\mathbb{R}^{d\times d}$ can be written in
the form
\[
H=\sum_{i=1}^{d}\lambda_{i}e_{i}e_{i}^{\ast}\quad\text{(the spectral theorem
for symmetric matrices)}%
\]
where $\lambda_{i}\in\mathbb{R}$, $\left\{  e_{1},e_{2},...,e_{d}\right\}  $
is an orthonormal basis of $\mathbb{R}^{d}$ and $He_{i}=\lambda_{i}e_{i}$. If
we assume $H$ is a positive matrix, that is, there exist constants
$a_{H},b_{H}>0$ such that for all $u\in\mathbb{R}^{d}$%
\begin{equation}
a_{H}\left\vert u\right\vert ^{2}\leq\left\langle Hu,u\right\rangle \leq
b_{H}\left\vert u\right\vert ^{2}, \label{Hpdef}%
\end{equation}
then $a_{H}\leq\lambda_{i}\leq b_{H}$ for all $i\in\overline{1,d}$ and $H$ is
a nonsingular matrix since $a_{H}^{d}\leq\det\left(  H\right)  \leq b_{H}^{d}%
$. Moreover, one can obtain%
\[
\left\{
\begin{array}
[c]{l}%
a_{H}\leq\min\limits_{i\in\overline{1,d}}\lambda_{i}\leq\left\Vert
H\right\Vert _{op}=\max\limits_{i\in\overline{1,d}}\lambda_{i}\leq b_{H}%
~\quad\text{and}\medskip\\
\left\Vert H\right\Vert _{op}\leq\left\vert H\right\vert :=\sqrt
{\mathbf{Tr}\left(  HH_{t}^{\ast}\right)  }=\left(
%TCIMACRO{\tsum \nolimits_{i=1}^{d}}%
%BeginExpansion
{\textstyle\sum\nolimits_{i=1}^{d}}
%EndExpansion
\lambda_{i}^{2}\right)  ^{1/2}\leq\sqrt{d}\left\Vert H\right\Vert _{op}~.
\end{array}
\right.
\]
Since%
\[
H^{1/2}=\sum_{i=1}^{d}\sqrt{\lambda_{i}}e_{i}e_{i}^{\ast}\,,\quad H^{-1}%
=\sum_{i=1}^{d}\dfrac{1}{\lambda_{i}}e_{i}e_{i}^{\ast},\quad\text{and }\quad
H^{-1/2}=\sum_{i=1}^{d}\dfrac{1}{\sqrt{\lambda_{i}}}e_{i}e_{i}^{\ast}~,
\]
we also have%
\begin{equation}
\left\{
\begin{array}
[c]{l}%
\dfrac{1}{b_{H}}\leq\left\vert \left\vert H^{-1}\right\vert \right\vert
_{op}\leq\dfrac{1}{a_{H}},\quad\sqrt{a_{H}}\leq\left\vert \left\vert
H^{1/2}\right\vert \right\vert _{op}\leq\sqrt{b_{H}},\quad\dfrac{1}%
{\sqrt{b_{H}}}\leq\left\vert \left\vert H^{-1/2}\right\vert \right\vert
_{op}\leq\dfrac{1}{\sqrt{a_{H}}}\quad\text{and}\medskip\\
\mathbf{Tr}\left(  H^{-1}\right)  =%
%TCIMACRO{\dsum _{i=1}^{d}}%
%BeginExpansion
{\displaystyle\sum_{i=1}^{d}}
%EndExpansion
\dfrac{1}{\lambda_{i}}\geq\dfrac{d}{b_{H}}\geq\dfrac{1}{b_{H}}.
\end{array}
\right.  \label{trH_1}%
\end{equation}

\subsection{Gronwall's inequality}

\begin{lemma}
\label{Gronwall}\textit{Let }$\Theta,K,V:\left[  0,T\right]  \rightarrow
\mathbb{R}$ \textit{be bounded c\`{a}dl\`{a}g functions such that }$K_{\cdot
}\in BV\left(  \left[  0,T\right]  ;\mathbb{R}\right)  $ and $V$ is
\emph{continuous} and nondecreasing. \textit{If, for all }$0\leq s\leq T,$%
\begin{equation}
\Theta_{s}\leq\Theta_{T}+%
%TCIMACRO{\dint _{s+}^{T}}%
%BeginExpansion
{\displaystyle\int_{s+}^{T}}
%EndExpansion
\left[  dK_{r}+\Theta_{r}dV_{r}\right]  \label{g1}%
\end{equation}
\textit{then}%
\begin{equation}
\left\{
\begin{array}
[c]{ll}%
\left(  i\right)  & e^{V_{s}}\Theta_{s}\leq\Theta_{T}e^{V_{T}}+{%
%TCIMACRO{\dint _{s+}^{T}}%
%BeginExpansion
{\displaystyle\int_{s+}^{T}}
%EndExpansion
}e^{V_{r}}dK_{r}\quad\text{and}\medskip\\
\left(  ii\right)  & \Theta_{s}\leq e^{V_{T}-V_{s}}\left[  \Theta
_{T}+\left\Vert K\right\Vert _{BV\left(  \left[  0,T\right]  ;\mathbb{R}%
\right)  }\right]  .
\end{array}
\right.  \label{g3}%
\end{equation}
\textit{If, moreover, }$a\geq0$\textit{ is a constant and, for all }$0\leq
s\leq T,$ $\Theta_{s}\leq a+%
%TCIMACRO{\tint _{s+}^{T}}%
%BeginExpansion
{\textstyle\int_{s+}^{T}}
%EndExpansion
\Theta_{r}dV_{r}~$, \textit{then}%
\begin{equation}
\Theta_{s}\leq a\,e^{V_{T}-V_{s}}~. \label{g4}%
\end{equation}

\end{lemma}

\begin{proof}
Let $G_{s}=\Theta_{s}+%
%TCIMACRO{\tint _{0+}^{s}}%
%BeginExpansion
{\textstyle\int_{0+}^{s}}
%EndExpansion
\left(  dK_{r}+\Theta_{r}dV_{r}\right)  .$ Hypothesis (\ref{g1}) yields
$G_{s}\leq G_{T}~$, for all $s\in\left[  0,T\right]  $. Consequently,%
\[
d\left[  e^{V_{r}}\left(  G_{r}-\Theta_{r}\right)  \right]  =e^{V_{r}}\left(
G_{r}-\Theta_{r}\right)  dV_{r}+e^{V_{r}}\left(  dK_{r}+\Theta_{r}%
dV_{r}\right)  \leq G_{T}e^{V_{r}}dV_{r}+e^{V_{r}}dK_{r}~.
\]
Integrating from $s_{+}$ to $T$, we get%
\[
e^{V_{T}}\left(  G_{T}-\Theta_{T}\right)  -e^{V_{s}}\left[  G_{s}-\Theta
_{s}\right]  \leq G_{T}\left(  e^{V_{T}}-e^{V_{s}}\right)  +%
%TCIMACRO{\dint _{s+}^{T}}%
%BeginExpansion
{\displaystyle\int_{s+}^{T}}
%EndExpansion
e^{V_{r}}dK_{r}~.
\]
Hence%
\[
e^{V_{s}}\Theta_{s}\leq e^{V_{T}}\Theta_{T}+e^{V_{s}}\left(  G_{s}%
-G_{T}\right)  +%
%TCIMACRO{\dint _{s+}^{T}}%
%BeginExpansion
{\displaystyle\int_{s+}^{T}}
%EndExpansion
e^{V_{r}}dK_{r}\leq e^{V_{T}}\Theta_{T}+%
%TCIMACRO{\dint _{s+}^{T}}%
%BeginExpansion
{\displaystyle\int_{s+}^{T}}
%EndExpansion
e^{V_{r}}dK_{r}\leq e^{V_{T}}\left[  \Theta_{T}+\left\Vert K\right\Vert
_{BV\left(  \left[  0,T\right]  ;\mathbb{R}\right)  }\right]  ,
\]
since $e^{V_{r}}\leq e^{V_{T}}$ for $r\leq T$, $V$ being nondecreasing.
Dividing by $e^{V_{s}}>0$ gives ((\ref{g3})-$\left(  ii\right)  $%
).$\smallskip$

Let us prove (\ref{g4}). Since, from the assumptions, $\Theta_{T}\leq a$, then
(\ref{g4}) is a particular case of (\ref{g3}), with $K_{t}=\left(
a-\Theta_{T}\right)  \mathbf{1}_{\left\{  T\right\}  }\left(  t\right)  $. In
this situation, $\left\Vert K\right\Vert _{BV\left(  \left[  0,T\right]
;\mathbb{R}\right)  }=a-\Theta_{T}$ and the proof is now complete.\hfill
\end{proof}

\subsection{Regularization of convex functions}

We now present some classical instruments used when we need to approximate a
maximal monotone operator given by the subdifferential of a proper convex
lower semicontinuous function. For more details, the interested reader can
consult Pardoux and R\u{a}\c{s}canu \cite[Section 6.3.7]{Pardoux/Rascanu:14}.

\begin{lemma}
\label{conv}Let $\varphi:\mathbb{R}^{d}\rightarrow(-\infty,+\infty]$ be a
proper lower semicontinuous convex function and let $\varphi_{\varepsilon}$ be
its Moreau regularization,
\[
\varphi_{\varepsilon}(x):=\inf\left\{  \frac{1}{2\varepsilon}|z-x|^{2}%
+\varphi(z):z\in\mathbb{R}^{d}\right\}  ,
\]
where $\varepsilon>0.$ Then for all $x,y\in\mathbb{R}^{d}$ and every
$\varepsilon,\delta>0$, the following are true:

\begin{itemize}
\item[$\left(  a\right)  $] $\varphi_{\varepsilon}$ is a convex function of
class $C^{1}$ on $\mathbb{R}^{d}$, and $\nabla\varphi_{\varepsilon}$ is
Lipschitz on $\mathbb{R}^{d}$ with Lipschitz constant $\varepsilon^{-1};$

\item[$\left(  b\right)  $] $\nabla\varphi_{\varepsilon}\left(  x\right)
=\dfrac{1}{\varepsilon}(x-J_{\varepsilon}x)$, where $J_{\varepsilon
}x=(I+\varepsilon\partial\varphi)^{-1}(x);$

\item[$\left(  c\right)  $] $\varphi_{\varepsilon}(x)=\dfrac{1}{2\varepsilon
}|x-J_{\varepsilon}x|^{2}+\varphi(J_{\varepsilon}x);$

\item[$\left(  d\right)  $] $\nabla\varphi_{\varepsilon}(x)\in\partial
\varphi(J_{\varepsilon}x);$

\item[$\left(  e\right)  $] $\left\langle \nabla\varphi_{\varepsilon
}(x)-\nabla\varphi_{\delta}(y),x-y\right\rangle \geq-\left(  \varepsilon
+\delta\right)  |\nabla\varphi_{\varepsilon}(x)||\nabla\varphi_{\delta}(y)|;$

\item[$\left(  f\right)  $] if $\left(  u_{0},\hat{u}_{0}\right)  \in
\partial\varphi$ then

\begin{itemize}
\item[$\left(  f_{1}\right)  $] $\left\vert \nabla\varphi_{\varepsilon}\left(
u_{0}\right)  \right\vert \leq\left\vert \hat{u}_{0}\right\vert ,$

\item[$\left(  f_{2}\right)  $] $\left\vert J_{\varepsilon}\left(  y\right)
\right\vert \leq\left\vert y-u_{0}\right\vert +\varepsilon\left\vert \hat
{u}_{0}\right\vert +\left\vert u_{0}\right\vert $ for all $y\in\mathbb{R}%
^{d},$

\item[$\left(  f_{3}\right)  $] $\varphi_{\varepsilon}\left(  y\right)
\geq\varphi\left(  J_{\varepsilon}y\right)  \geq\varphi\left(  u_{0}\right)
-\left\vert \hat{u}_{0}\right\vert \left\vert y-u_{0}\right\vert
-\varepsilon\left\vert \hat{u}_{0}\right\vert ^{2}$ for all $y\in
\mathbb{R}^{d}$,

\item[$\left(  f_{4}\right)  $] $\big|\varphi\left(  J_{\varepsilon}y\right)
-\varphi\left(  u_{0}\right)  \big|\leq\big\langle\nabla\varphi_{\varepsilon
}\left(  y\right)  ,J_{\varepsilon}y-u_{0}\big\rangle+2\left\vert \hat{u}%
_{0}\right\vert \left\vert y-u_{0}\right\vert +2\varepsilon\left\vert \hat
{u}_{0}\right\vert ^{2}$ for all $y\in\mathbb{R}^{d}$.
\end{itemize}
\end{itemize}
\end{lemma}

\bigskip

\noindent Funding: Not applicable.\bigskip

\begin{acknowledgement}
The authors would like to express their sincere gratitude to the anonymous
referees for their comments and suggestions, which have resulted in
considerable improvement of the results and presentation of this article.
\end{acknowledgement}

\end{document}